\documentclass
[
11pt]
{amsart}%

\usepackage{cite}
\usepackage{amssymb}
\usepackage{amsmath}
\usepackage{amsthm}
\usepackage{amsfonts}
\usepackage{bbm}

\usepackage[toc,page]{appendix}
\usepackage{subfigure}
\usepackage{mathtools}
\usepackage{bm}
\usepackage{comment}

\usepackage{enumerate,esint,multicol, mathptmx}
\usepackage[colorlinks=true, pdfstartview=FitV, linkcolor=blue,
            citecolor=blue, urlcolor=blue]{hyperref}
\usepackage[usenames]{color}
\definecolor{Red}{rgb}{0.7,0,0.1}
\definecolor{Green}{rgb}{0,0.7,0}
\usepackage{accents}
\usepackage{graphicx}
\usepackage{xcolor, geometry}
\usepackage[capitalize,nameinlink,noabbrev]{cleveref}

\numberwithin{equation}{section}

\newtheorem{Theorem}{Theorem}[section]
\newtheorem{Proposition}[Theorem]{Proposition}
\newtheorem{Lemma}[Theorem]{Lemma}
\newtheorem{Corollary}[Theorem]{Corollary}

\theoremstyle{definition}

\newtheorem{Example}{Example}[section]
\newtheorem{Assumption}{Assumption}
\newtheorem{Not}{Notation}[section]

\makeatother\newtheorem{Remark}[Theorem]{Remark}

\newcommand{\PP}{\mathbf{P}}

\newcommand{\N}{\mathbf{N}}

\newcommand{\E}{\mathbf{E}}

\newcommand{\RR}{\mathbf{R}}

\begin{document}

\title[Functional inequalities for diffusions with degenerate noise]{Functional inequalities for a family of infinite-dimensional diffusions with degenerate noise}

\author[Baudoin]{Fabrice Baudoin{$^{\ast}$}}
\thanks{\footnotemark {$\ast$} This research was supported in part by NSF grant DMS-\textcolor[rgb]{1.00,0.00,0.00}{ DMS-2247117}.}
\address{Department of Mathematics\\
University of Connecticut\\
Storrs, CT 06269, USA} \email{fabrice.baudoin@uconn.edu}

\author[Gordina]{Maria Gordina{$^{\ast\ast}$}}
\thanks{\footnotemark {$\ast\ast$} Research was supported in part by NSF grant  \textcolor[rgb]{1.00,0.00,0.00}{DMS-2246549}. The author acknowledges the support by the Hausdorff Center of Mathematics (Bonn, Germany) and the IHES (France), where parts of the work were completed.}
\address{Department of Mathematics\\
University of Connecticut\\
Storrs, CT 06269, USA} \email{maria.gordina@uconn.edu}

\author[Herzog]{David P. Herzog{$^{\ddag}$}}
\thanks{\footnotemark {$\ddag$} Research was supported in part by NSF grants \textcolor[rgb]{1.00,0.00,0.00}{DMS-1612898} and \textcolor[rgb]{1.00,0.00,0.00}{DMS-1855504}}
\address{Department of Mathematics\\Iowa State University\\Ames, IA 50311, U.S.A.} \email{dherzog@iastate.edu}

\author[Kim]{Jina Kim{$^{\ddag}$}}
\thanks{\footnotemark {$\ddag$} Research was supported in part by NSF grants \textcolor[rgb]{1.00,0.00,0.00}{DMS-1612898} and \textcolor[rgb]{1.00,0.00,0.00}{DMS-1855504}}
\address{Department of Mathematics\\Trinity University\\San Antonio, TX 78212, U.S.A.} \email{jkim7@trinity.edu}

\author[Melcher]{Tai Melcher{$^{\dagger}$}}
\thanks{\footnotemark {$\dagger$} Research was supported in part by NSF grant \textcolor[rgb]{1.00,0.00,0.00}{DMS-1255574}.}
\address{Department of Mathematics\\
University of Virginia \\
Charlottesville, VA 22903, USA} \email{melcher@virginia.edu}


\begin{abstract}
For a family of infinite-dimensional diffusions with degenerate noise, we develop a modified $\Gamma$ calculus on finite-dimensional projections of the equation in order to produce explicit functional inequalities that can be scaled to infinite dimensions.  The choice of our $\Gamma$ operator appears canonical in our context, as the estimates depend only on the induced control distance.  We apply the general analysis to a number of examples, exploring implications for quasi-invariance and uniqueness of stationary distributions.      
\end{abstract}


\maketitle

\tableofcontents


\renewcommand{\contentsname}{Table of Contents}

\maketitle

\section{Introduction}

Understanding when measures on infinite-dimensional spaces possess smoothness properties reminiscent of those on $\RR^n$ is a fundamental problem in the theory of diffusion processes.  In the finite-dimensional context of $\RR^n$, H\"{o}rmander's bracket generating condition~\cite{Hormander1967a} provides the criteria to determine when the law of the diffusion with smooth $(C^\infty)$ coefficients has a smooth  density with respect to Lebesgue measure.  At the level of the associated diffusion, H\"{o}rmander's condition translates to how the \emph{external} randomness present in the equation propagates \emph{internally} to produce a distribution with $C^\infty$ density.  Thus, provided there is sufficient randomness in the equation to ensure the needed propagation, the existence and smoothness of the density follows from H\"{o}rmander's result.  

While the mantra of \emph{sufficient noise implies smoothness of the law} applies readily in the finite-dimensional setting via H\"{o}rmander's classical hypoellipticity result, the story in infinite dimensions is more nuanced. For example, even in relatively simple settings when noise acts on every basis direction in an infinite-dimensional Hilbert space, which in particular acts as a \emph{natural} phase space for the stochastic solution, the laws of two solutions started from close initial conditions may be mutually singular at any fixed time $t>0$.  This, in turn, implies that the associated Markov semigroup is not strong Feller~\cite[Example 3.15]{HM_06}. Furthermore, while H\"{o}rmander's condition in finite dimensions allows for the development of regularity theory for the corresponding PDEs even in the absence of ellipticity, such PDE techniques are not available in infinite dimensions. For example, even in the context of an infinite-dimensional Brownian motion, the classical Harnack inequality enjoyed by similar, non-degenerate finite-dimensional processes fails to hold~\cite{BassGordina2012}.

Because of the prevalence of degenerate noises in applications and these nuances, understanding what hypoellipticity means for infinite-dimensional diffusions continues to be an active area of research.  On the one hand, significant progress has been made in this direction in the context of fluid mechanics.  Indeed, building off of the pioneering works on unique ergodicity of stochastically-forced PDEs~\cite{BKL_01, DD_03, FM_95, KS_01} as well as the known behavior of the finite-dimensional Galerkin approximations~\cite{EMattingly2001, Rom_04}, Hairer and Mattingly introduced the notion of \emph{asymptotic strong Feller property} in order to prove unique ergodicity of the two-dimensional Navier--Stokes equations on the period box under highly degenerate stochastic forcing~\cite{HM_06}.  This property was further developed and investigated in the works~\cite{HairerMattingly2011b, HMS_11, MP_06}.  In this context, there is just enough smoothing \emph{at time infinity}, as defined by the asymptotic strong Feller property, to conclude uniqueness of steady states.  We refer also to the work~\cite{FGHRW_16} which studies the Boussinesq equations under degenerate stochastic forcing, validating the asymptotic strong Feller property in that context. 

On the other hand, smoothness of measures in infinite dimensions can be understood as quasi-invariance under transformations such as translations. This allows for the definition of smoothness even in the absence of a natural reference measure such as Lebesgue measure. This point of view was pioneered by Malliavin in \cite{Malliavin1990a} leading to development of Malliavin calculus. More relevant to the current paper is the connection between hypoellipticity in infinite dimensions to quasi-invariance properties and their relation to classical functional inequalities, especially in the context of heat kernel measures on infinite-dimensional Heisenberg-like groups~\cite{BaudoinGordinaMelcher2013, BaudoinGordinaMariano2020, BassGordina2012, DG_08, DG_09, Gor_03}.  In these settings, the noise structure in the equations is different than in the fluid models above, as the driving external randomness is infinite-dimensional as opposed to acting on a few low frequencies in Fourier space.  Furthermore, techniques from Dirichlet forms, classical Cameron-Martin-Girsanov theorem and $\Gamma$ calculus are often employed in place of Harris' theorem and Malliavin calculus in the fluids setting.  We refer also to~\cite{DL_95, FR_00} for related work.

The goal of this paper is to make progress on understanding the meaning of hypoellipticity in infinite-dimensions by developing a modified $\Gamma$ calculus. We build off of understanding from previous work in~\cite{Baudoin2017a, BaudoinGordinaHerzog2021, BaudoinGordinaMariano2020, BGM_23}.  Specifically, we design a modified $\Gamma$ operator in order to obtain classical functional inequalities, e.g. Wang--Harnack and reverse log--Sobolev, for a class of infinite-dimensional diffusions arising as solutions to certain stochastic differential equations. These are generalizations of the Kolmogorov diffusion studied in~\cite{Baudoin2017a, BaudoinGordinaHerzog2021, BaudoinGordinaMariano2020, BGM_23}. An important contribution of this work is that the choice of our modified $\Gamma$ operator appears canonical, as the constants in the functional inequality bounds are independent of both the noise and spatial dimension for finite-dimensional projections of the equation.  However, the dependence on these parameters is intrinsic in the induced distance, which we show can be estimated in a variety of examples. The existence of a natural notion of distance here is notable, as previously there has been no geometry in which to work with this class of distributions. This is in contrast to other hypoelliptic models (like the Heisenberg group) where there is a natural geometric framework coming from the sub-Riemannian distance. For this class of diffusions, we also study large-time properties when the structure allows for it, developing a criteria for mutual absolute continuity of stationary distributions for the associated Markov semigroup.  This is done by using the deduced Wang-Harnack type inequality in finite dimensions, and scaling it appropriately to infinite dimensions.

The organization of this paper is as follows.  In \cref{sec:FD}, we introduce the finite-dimensional setting which will later be scaled to infinite dimensions in~\cref{sec:FDTID}.  In~\cref{sec:modgrad}, we develop our modified $\Gamma$ calculus in the context of the finite-dimensional setting.  In particular, we derive our choice of $\Gamma$ operator and deduce a number of functional inequalities based on this choice.  In \cref{sec:class}, we estimate the control distance associated to our choice of $\Gamma$ operator in several concrete examples.  Based on the derivations of these functional inequalities, the control distance is the only term one has left to estimate to produce fully explicit estimates.  Finally in~\cref{sec:FDTID}, we scale the functional inequalities and the finite-dimensional setting to infinite dimensions, obtaining criteria for quasi-invariance as well as mutual absolute continuity of invariant probability measures.  At the end of \cref{sec:FDTID}, we revisit some of the examples discussed in \cref{sec:class}, applying the results obtained in this section.

\section{The finite-dimensional setting}
\label{sec:FD}
\subsection{The main equation and hypoellipticity}  Let $\mathcal{B}$ denote the Borel sigma field of subsets of $\RR^n$.  In the finite-dimensional setting, we consider the following stochastic differential equation (SDE) on $\RR^n$
\begin{align}
\label{eqn:SDE}
d x_t = Ax_t  \, dt + \sigma \, dB_t,
\end{align}
where $A$ and $\sigma$ are $n\times n$ real matrices, and $B_t$ is a standard, $n$-dimensional Brownian motion defined on a probability space $(\Omega, \mathcal{F}, \PP, \E)$. For all initial conditions $x\in \RR^n$, equation~\eqref{eqn:SDE} has a unique pathwise solution $x_t(x)$, defined for all times $t\geqslant 0$, which can be explicitly written as
\begin{align}
\label{eqn:gaussian}
x_t (x)= e^{tA} x + \int_0^t e^{(t-s) A} \sigma dB_s .
\end{align}
Unless we must emphasize the initial condition, we will write $x_t$ as shorthand notation for a generic solution of~\eqref{eqn:SDE}.

Solutions of \eqref{eqn:SDE} are Markovian and we let $\{ P_t \}_{t\geqslant 0}$ denote the corresponding Markov semigroup.  We recall that $\{ P_t \}_{t\geqslant 0}$ acts on bounded, $\mathcal{B}$-measurable functions $f: \RR^n\rightarrow \RR$ by
\begin{align*}
P_t f(x) := \E f(x_t(x)), \,\, t\geqslant 0,
\end{align*}
and acts dually on a probability measure $\nu$ on $\left( \RR^{n}, \mathcal{B}\right)$  via
\begin{align*}
\nu P_t(A):= \int_{\RR^n} \nu(dx) P_t \mathbf{1}_A(x), \,\, \,\,t \geqslant 0, \,\, A\in \mathcal{B}.
\end{align*}
A probability measure $\nu$ on $\left( \RR^{n}, \mathcal{B}\right)$ is called a \emph{stationary distribution} if $\nu P_t =\nu$ for all $t\geqslant 0$.
 For $x\in \RR^n$, $t\geqslant 0$ and $A\in \mathcal{B}$, we let
\begin{align*}
P_t(x, A): = P_t \mathbf{1}_A(x)= \PP \{ x_t(x) \in A \}
\end{align*}
denote the \emph{Markov transition probability} associated to $\{ P_t \}_{t\geqslant 0}$.  Throughout, $L$ will denote the following second-order operator
\begin{align}
\label{eqn:gen}
L = \sum_{i=1}^n (A x)_i   \frac{\partial}{\partial x_i}  + \frac{1}{2}\sum_{i,j=1}^n (\sigma \sigma^{\ast})_{ij} \frac{\partial^2}{\partial x_i \partial x_j}.
\end{align}
Note that $L$ corresponds to the action of the infinitesimal generator of $\{ P_t\}_{t\geqslant 0}$ on a domain of sufficiently smooth functions, e.g. $C^2$ functions $f:\RR^n\rightarrow \RR$ with compact support.  We offer the slight abuse of terminology and refer to $L$ as the generator of the Markov process $x_t$ throughout.

We are interested in the case when the noise in equation~\eqref{eqn:SDE} is \emph{degenerate}, i.e. $\text{rank} (\sigma) < n$, but the process $x_t$ has a transition probability density function $p_t(x,y)$ respect to Lebesgue measure which is $C^\infty$ for all $(t,x,y) \in (0, \infty) \times \RR^n \times \RR^n$.  When $\text{rank}(\sigma)< n$, the existence and smoothness of the transition density, which will be referred to throughout as \emph{hypoellipticity}, is not immediate precisely because the noise is degenerate.  However, in this context, hypoellipticity can be established under further conditions on the interaction between $A$ and $\sigma$ in essentially two ways.  The first and perhaps most utilized way is to apply H\"{o}rmander's hypoellipticity theorem~\cite{Hormander1967a} (see also~\cite{StroockPDEBook2012}).  However, H\"{o}rmander's result is more powerful than needed in the context~\eqref{eqn:SDE}.  For our purposes, a more direct way is to verify the \emph{Kalman rank condition}~\cite{KalmanFalbArbibBook1969} (see \cref{assump:1} below) and show that this condition implies hypoellipticity.  The Kalman rank condition is usually employed to ensure controllability of the resulting ordinary differential equation (ODE) when the independent Brownian motions in~\eqref{eqn:SDE} are replaced by deterministic controls.
\begin{Assumption}
\label{assump:1}
The \emph{Kalman rank  condition} is satisfied; that is, if we define the matrix
\begin{align*}
A_\sigma:=\begin{bmatrix}
\sigma & A\sigma & A^2\sigma & \ldots &A^{n-1} \sigma
\end{bmatrix}, \end{align*}
then
\begin{align*}
\text{rank}(A_\sigma)=n.
\end{align*}
\end{Assumption}
Consider the mean $m_t(x)$ and covariance $\Sigma_t$ of the process~\eqref{eqn:gaussian} given by
\begin{align}
\label{eqn:cov}
m_t(x) = e^{tA} x \qquad \text{ and } \qquad
\Sigma_t = \int_0^t e^{(t-s) A} \sigma \sigma^{\ast} e^{(t-s) A^{\ast}} \, ds.
\end{align}
A short argument (see \cref{lem:pd} below) shows that \cref{assump:1} implies that $\Sigma_t$ as in~\eqref{eqn:cov} is invertible for all $t>0$.  This  in turn implies that the process $x_t$ is multivariate Gaussian with an explicit transition density given by
\begin{align}
\label{eqn:tdens}
&p_t(x,y)= \frac{1}{\det(2\pi \Sigma_t)^{1/2}} \exp\big(-\tfrac{1}{2} \langle \Sigma_t^{-1} (y- m_t(x)), y- m_t(x)\rangle  \big), \,\,\, x,y \in \RR^n, \, t>0,
\end{align}
where $\langle\cdot, \cdot \rangle$ denotes the standard inner product on $\RR^n$.  Smoothness of $p_t(x,y)$ for $(t,x,y) \in (0, \infty) \times \RR^n \times \RR^n$ can then be readily verified from~\eqref{eqn:cov} and~\eqref{eqn:tdens}.  It should be noted that the proof of \cref{lem:pd} is Malliavin's probabilistic proof of H\"{o}rmander's theorem~\cite{KusuokaStroock1987, Malliavin1978a, Norris1986} in the simplified context~\eqref{eqn:SDE}.

 \begin{Lemma}
\label{lem:pd}
Suppose that \cref{assump:1} is satisfied.  Then for every $t>0$, $\Sigma_t$ in~\eqref{eqn:cov} is invertible.
\end{Lemma}

\begin{proof}
Fix $t>0$ and suppose that there exists $x\in \RR^n$ such that $\langle \Sigma_t x, x\rangle =0$.  This implies
\begin{align*}
0=\langle \Sigma_t x, x \rangle =  \int_0^t | \sigma^{\ast} e^{s A^{\ast}} x|^2 \, ds.
\end{align*}
Hence, by continuity, $\sigma^{\ast} e^{s A^{\ast}} x=0$ for all $s\in [0, t]$.  Differentiating $\sigma^{\ast} e^{sA^{\ast}}x$ repeatedly with respect to $s$ and evaluating at $s=0$ implies that
\begin{align*}
 \sigma^{\ast} (A^{\ast})^m x =0 \text{ for all } m =0,1,2, \ldots
\end{align*}
 Note by~\cref{assump:1} this implies $x=0$ since the rank of the matrix and its transpose are the same. \end{proof}

 \subsection{Elements of Bakry-\'{E}mery calculus}
In \cref{sec:modgrad}, we provide a framework for establishing classical functional inequalities in the finite-dimensional hypoelliptic setting~\eqref{eqn:SDE}, where the constants in the bounds are independent of the spatial dimension $n$.  Importantly, we employ a modification of the Bakry-\'{E}mery calculus~\cite{BakryEmery1985} by proposing a \emph{natural} gradient associated to the system~\eqref{eqn:SDE} which plays the role analogous to the usual carr\'{e} du champ in the uniformly elliptic setting, e.g when $\text{rank}(\sigma) =n$ in equation~\eqref{eqn:SDE}~\cite{BakryEmery1985}.  Here, we recall some of the basic elements of the Bakry-\'{E}mery calculus.

To the generator $L$ as in~\eqref{eqn:gen}, we associate the \emph{carr\'{e} du champ} $\Gamma$ and its \emph{iterate} $\Gamma_2$, which
for $f,g \in C^\infty(\RR^n; \RR)$ are defined by
\begin{align*}
\Gamma(f,g) := \tfrac{1}{2}L(fg) - &\tfrac{1}{2}g Lf-\tfrac{1}{2}f Lg, \qquad \Gamma(f):= \Gamma(f,f),\\
\Gamma_2(f)& := \tfrac{1}{2}L \Gamma(f)- \Gamma(f, Lf) .
\end{align*}
Importantly, $\Gamma$ and $\Gamma_2$ arise naturally in relation to the Markov semigroup $\{ P_t\}_{t\geqslant 0}$.  If, for example, the Markov process $x_t$ has a stationary distribution $\nu$ and $f :\RR^n\rightarrow \RR$ is bounded measurable, then a formal calculation using stationarity of $\nu$ gives
\begin{align}
\label{eqn:L2norm}
\frac{1}{2}\frac{d}{dt}\| P_t f\|_{L^2(\nu)}^2 &= \int_{\RR^n} (P_t f) L P_t f \, d\nu = - \int_{\RR^n} \Gamma(P_t f) \,d\nu,\\
\notag
\frac{1}{2}\frac{d}{dt} \int_{\RR^n} \Gamma(P_t f) \, d\nu &= \int_{\RR^n} \Gamma(L P_t f, P_t f) \, d\nu =- \int_{\RR^n} \Gamma_2(P_t f) \, d\nu.
\end{align}
In particular, the forms $\Gamma, \Gamma_2$ are the resulting objects that arise from iterating time with respect to the norm $\| \cdot \|_{L^2(\nu)}$.  This calculation has been exploited in several situations in order to study convergence to equilibrium for SDEs with an explicit stationary distribution $\nu$~\cite{Baudoin2017a, BakryEmery1985, BaudoinGordinaHerzog2021, DolbeaultMouhotSchmeiser2015, HerauNier2004, GrothausStilgenbauer2015, Talay2002}.  Such a function was also used to prove a logarithmic Sobolev inequality in \cite{BakryEmery1986}, and this approach has been employed to prove different functional inequalities as described in the monograph \cite{BakryGentilLedouxBook}.

Crucially in what follows, we employ a generalization of $\Gamma$ and $\Gamma_2$. Recalling that $\langle \cdot, \cdot \rangle$ denotes the standard inner product on $\RR^n$, for any $n\times n$ real matrix $G$ and $f,g\in C^\infty(\RR^n; \RR)$, we introduce the notation
\begin{align}\label{eqn:gammanot}
&\Gamma^G(f,g) =  \langle G \nabla f , \nabla g\rangle,\qquad
\Gamma^G(f)=\Gamma^G(f,f),
\\
\label{eqn:gammanot1}& \Gamma_2^G(f) = \frac{1}{2}L \Gamma^G(f)- \Gamma^G(f, Lf).
\end{align}
We see that $\Gamma^G$ is a generalization of $\Gamma$ since
\begin{align*}
\Gamma(f,g) = \frac{1}{2}\sum_{\ell, m=1}^n (\sigma \sigma^{\ast})_{\ell m} \partial_\ell f \partial_m g = \Gamma^{\tfrac{\sigma \sigma^{\ast}}{2}}(f,g).
\end{align*}
Similar functionals with a time-independent $G$ were considered in \cite{baudoin2016wasserstein, Monmarche2023}. Below, we find a convenient way to define a time-dependent, \emph{natural} $G$ associated to $x_t$ from which the desired functional inequalities can be obtained and such that the constants in the bounds of these inequalities are independent of the dimension $n$.  Although we will not always need to assume it, $G$ should be thought of as symmetric, positive-definite so that $\Gamma^G(f)$ is equivalent to $| \nabla f|^2$, and so can be interpreted as changing the underlying metric.

%
%
%
%

\section{Functional inequalities by modified gradients}
\label{sec:modgrad}

The goal of this section is to construct an appropriate, time-dependent matrix $G$ so that, by using $\Gamma^G$ and $\Gamma_2^G$ defined in~\eqref{eqn:gammanot} and~\eqref{eqn:gammanot1}, we arrive at various functional inequalities where the constants in the inequalities are independent of the spatial dimension $n$.  We will see that by using the framework outlined below, we arrive at a natural choice for $G$.  Note that a similar strategy was employed using a diagonal matrix in the setting of the Kolmogorov diffusion in~\cite{BaudoinGordinaMariano2020, BaudoinGordinaMelcher2023}.  Our construction holds in the more general setting~\eqref{eqn:SDE} for a convenient symmetric positive-definite matrix $G$.  We also remark that a similar construction was used in the paper~\cite{Menegaki2020} in the context of chain of oscillators, but the matrix used there is not time-dependent. We finally point out \cite{MR4444114} where dimension dependent Hardy-Littlewood-Sobolev  inequalities for similar types of operators are obtained using heat kernel estimates.

We begin by computing $\Gamma_2^G$ for a general, symmetric matrix $G$.  Here and in what follows, all matrices below are assumed to be spatially constant.

\begin{Proposition}\label{prop:2}
Suppose that $G$ is an $n\times n$ symmetric matrix.  Then for all $f\in C^\infty(\RR^n; \RR)$, we have the formula
\begin{align}
\label{eqn:gamma2}\Gamma_2^G(f)= - \Gamma^{AG}(f)+ \frac{1}{2}\sum_{\ell} \Gamma^G( (\sigma^{\ast} \nabla f)_\ell).
\end{align}
If we furthermore assume that $G$ is nonnegative-definite, then
 \begin{align}
 \label{eqn:gamma2bound}
\Gamma_2^G(f)\geqslant - \Gamma^{AG}(f)
\end{align}
for all $f\in C^\infty(\RR^n; \RR)$.
\end{Proposition}

\begin{proof}
	Let $f\in C^\infty(\RR^n; \RR)$.
Standard calculations using symmetry of $G$ give
\begin{align*}
 L \Gamma^G(f) &= 2 \Gamma^G(f, Lf) - 2\Gamma^{A G}(f) + \sum_{\ell, m}(\sigma \sigma^{\ast})_{\ell m} \langle G \nabla (\partial_\ell f), \nabla (\partial_m f) \rangle.
 \end{align*}
Relation~\eqref{eqn:gamma2} then follows using matrix arithmetic on the last term on the right-hand side above.  Under the additional assumption that $G$ is nonnegative-definite, the inequality~\eqref{eqn:gamma2bound} follows immediately from~\eqref{eqn:gamma2}.
\end{proof}

To obtain a gradient estimate using a time-dependent matrix $G$ in $\Gamma^G$, we fix a horizon time $t>0$, and $f\in L^\infty(\RR^n;\RR)$ and consider the functional
\begin{align}\label{eqn:phidef}
\phi(s):=  P_s \Gamma^{G(s)}( P_{t-s} f), \,\,\, s\in [0, t),
\end{align}
for a differentiable matrix $s\mapsto G(s)$ on $[0, t]$ which is $n\times n$, symmetric.  Below, we see that the expression~\eqref{eqn:phidef} plays a role similar to the $L^2(\nu)$ norm in~\eqref{eqn:L2norm}.  The following lemma yields the expression which will, in turn, lead to the definition of a convenient $G$.
 \begin{Lemma}
 \label{lem:firstlem}
Fix $t>0$.  Suppose that $\cref{assump:1}$ is satisfied and that $s\mapsto G(s)\in C^1([0, t]; M_{n\times n})$ is symmetric and nonegative-definite on $[0, t]$.  For any $f\in  L^\infty(\RR^n; \RR)$, let $\phi$ be as in~\eqref{eqn:phidef}.  Then $\phi$ is defined and continuously differentiable on $(0,t)$ and, moreover, for $s\in (0,t)$ we have 
\begin{align}
\label{eqn:phip1}
\phi'(s)= P_s 2\Gamma_2^{G(s)} (P_{t-s} f) + P_s \Gamma^{G'(s)}(P_{t-s} f) \geqslant P_s \Gamma^{G'(s)-2AG(s)} (P_{t-s} f) .
\end{align}
 \end{Lemma}

\begin{Remark}
\label{rem:just}
The fact that $\phi$ as in~\eqref{eqn:phidef} makes sense, is continuously differentiable on $(0,t)$ and has \emph{nice} properties permitting the calculations leading to the proof of~\cref{lem:firstlem} follows from \cref{assump:1} and the explicit representation
\begin{align}
\label{eqn:repden}
P_t f(x)&= \frac{1}{\det(2\pi \Sigma_t)^{1/2}} \int_{\RR^n} f(y) \exp\big(-\tfrac{1}{2} \langle \Sigma_t^{-1} (y- m_t(x)), y- m_t(x)\rangle  \big)\, dy
\\
\label{eqn:repden1}&=\E f( \Sigma_t^{1/2} Z+ m_t(x)),
\end{align}
where $t>0$, $m_t(x)$ and $\Sigma_t$ are as in~\eqref{eqn:cov}, and $Z$ is the standard normal distribution on $\RR^n$.  Recalling that the covariance matrix $\Sigma_t$ is invertible for $t>0$ under \cref{assump:1} by \cref{lem:pd}, it follows from~\eqref{eqn:repden} that the semigroup $\{P_t\}_{t\geq 0}$ is strong Feller and Mehler.  Furthermore, one can use  formula~\eqref{eqn:repden} to explicitly find derivatives of $P_t f(x)$, provided $t>0$, to fully justify the regularity of $\phi$ on $(0,t)$ and the remaining derivative calculations in the proof of~\cref{lem:firstlem}.  Later we will also use formula~\eqref{eqn:repden1} to do similar calculations under further regularity hypotheses on the test function $f$.   
\end{Remark}

\begin{proof}[Proof of \cref{lem:firstlem}]
Following \cref{rem:just}, we observe that for $s\in (0,t)$ and $f\in L^\infty(\RR^n; \RR)$, symmetry of $G(s)$ implies
\begin{align*}
\phi^{\prime}(s)&= P_s L \Gamma^{G(s)} (P_{t-s} f) - 2P_s \Gamma^{G(s)}( P_{t-s} f, L P_{t-s} f) + P_s \Gamma^{G^{\prime}(s)}(P_{t-s} f)\\
&=P_s 2\Gamma_2^{G(s)}(P_{t-s} f) + P_s \Gamma^{G^{\prime}(s)} (P_{t-s} f).
\end{align*}
The bound in the result follows after applying~\cref{prop:2}.
\end{proof}

\subsection{Picking $G$ based on \cref{lem:firstlem}}
\label{sec:choice}

There are many workable choices for the matrix $G(s)$ corresponding to the dynamics~\eqref{eqn:SDE}.  Given the bound in relation~\eqref{eqn:phip1}, we argue  that the choice given  in the remainder of this section is canonical.  In particular, fixing a horizon time $t>0$, we define the matrix-valued function $s\mapsto G(s,t)$ on $[0,t]$ by
\begin{align}
\label{eqn:Gstdef}
G(s,t)= \int_s^t e^{(s-v) A}\sigma \sigma^{\ast} e^{(s-v)A^{\ast}} \, dv.
\end{align}
Note that this matrix satisfies the backward matrix-valued ODE
\begin{align*}
\begin{cases}
\partial_s G(s,t)= - \sigma \sigma^{\ast} + A G(s,t) + G(s,t) A^{\ast}, \, s\in [0,t]&
\\
G(t,t)=0.&
\end{cases}
\end{align*}
By symmetry of $G(s,t)$, for any $x\in \RR^n$ we have that
\begin{align*}
\langle \partial_s G(s,t) x, x \rangle = - \langle \sigma \sigma^{\ast} x, x \rangle + \langle 2 A G(s,t) x, x \rangle.
\end{align*}
We observe that picking $G$ in this way allows us to replace the right-hand side of the bound in~\eqref{eqn:phip1} with
\begin{align*}
- P_s \Gamma^{\sigma \sigma^{\ast}}(P_{t-s} f)=-2P_s \Gamma(P_{t-s}f)  = - \frac{d}{ds} P_s (P_{t-s} f)^2.
\end{align*}
As a consequence, we produce a reverse Poincar\'e-type inequality/gradient bound for the semigroup associated to the modified operator $\Gamma^G$.

\begin{Proposition}[Reverse Poincar\'e-type inequality]
\label{prop:L2smoothing}
Let $t>0$ and $s\mapsto G(s,t)$ be as in~\eqref{eqn:Gstdef} and suppose that~\cref{assump:1} is satisfied.  Then for all $f\in L^\infty(\RR^n; \RR)$ we have the explicit bound
\begin{align}
\label{eqn:smoothing}
\Gamma^{G(0,t)}(P_{t} f) \leqslant P_t f^2 - (P_t f)^2.
\end{align}
\end{Proposition}

\begin{Remark}
By a nearly identical argument to the one used in \cref{lem:pd}, under \cref{assump:1}, the matrix $G(0,t)$ is invertible for $t>0$.  Thus the modified gradient on the lefthand side of~\eqref{eqn:smoothing} is comparable to the usual gradient.
\end{Remark}

\begin{proof}[Proof of Proposition~\ref{prop:L2smoothing}]
By~\cref{assump:1} and a density argument, it suffices to prove estimate~\eqref{eqn:smoothing} for $f\in C^\infty_0(\RR^n; \RR)$ where $C_0^\infty(\RR^n;\RR)$ denotes the space of smooth real-valued functions with compact support in $\RR^n$. If we let $\phi(s)$ be as in~\eqref{eqn:phidef} with $G(s):=G(s,t)$, then \cref{lem:firstlem} implies
\begin{align}
\label{eqn:convergetoequil}
 \phi'(s)
&= 2 P_s \Gamma_2^{G(s)} (P_{t-s} f) + P_{s} \Gamma^{G'(s)}(P_{t-s} f)   \geqslant -2 P_s \Gamma(P_{t-s} f)
\end{align}
for $s\in (0, t)$.
Next, notice that for $s\in (0,t)$
\begin{align*}
\frac{d}{ds}P_s (P_{t-s} f)^2 = 2 P_s \Gamma(P_{t-s} f).  \end{align*}
Thus for $\epsilon >0$ sufficiently small, integrating both sides of~\eqref{eqn:convergetoequil} from $\epsilon$ to $t-\epsilon$ we arrive at the bound
\begin{align*}
P_\epsilon \Gamma^{G(\epsilon, t)} (P_{t-\epsilon} f) - P_{t-\epsilon} \Gamma^{G(t-\epsilon, t)} (P_\epsilon f) \leqslant P_{t-\epsilon} (P_\epsilon f)^2 - P_\epsilon (P_{t-\epsilon} f)^2 .
\end{align*}
In order to pass to the limit as $\epsilon \rightarrow 0$ in the above to arrive at~\eqref{eqn:smoothing} for $f\in C^\infty_0(\RR^n; \RR) $, it remains to show that
\begin{align*}
P_{t-\epsilon} \Gamma^{G(t-\epsilon, t)}(P_\epsilon f) \rightarrow 0\text{ as } \epsilon \rightarrow 0.
\end{align*}
Using~\eqref{eqn:repden1} for $t=\epsilon$, we find that
\begin{align*}
\Gamma^{G(t-\epsilon, t)} (P_\epsilon f) = \langle G(t-\epsilon, t) \E e^{\epsilon A^{\ast}} \nabla f( \Sigma_\epsilon^{1/2} Z + m_\epsilon (x)) , \E  e^{\epsilon A^{\ast}} \nabla f( \Sigma_\epsilon^{1/2} Z + m_\epsilon (x)) \rangle \end{align*}
is both uniformly bounded in $\epsilon$ and converges to $0$ as $\epsilon \rightarrow 0$.  This now establishes the bound~\eqref{eqn:smoothing} for $f\in C^\infty_0(\RR^n; \RR)$, finishing the proof.
\end{proof}

Because of the significance of the matrix $G(0,t)$, we introduce the following notation.

\begin{Not}
Throughout the remainder of the paper, we set
\begin{align}
\label{eqn:G_t}
G_t:=G(0,t)
\end{align}
where $G(s,t)$ is as in~\eqref{eqn:Gstdef}.
\end{Not}

\subsection{Reverse log-Sobolev and Wang-Harnack inequalities}
We now see that a similar calculation to the one used to prove the reverse Poincar\'{e} inequality in \cref{sec:choice} can also be used to establish a reverse log Sobolev inequality and the Wang-Harnack inequality.

To prove the log-Sobolev inequality, for computational ease, we introduce the spaces $C^\infty_\epsilon(\RR^n; [0, \infty))$, $\epsilon >0$, defined by
\begin{align*}
C_\epsilon^\infty(\RR^n; [0, \infty)) := \{ f \, : \, f = \varphi + \epsilon, \varphi \in C_0^\infty(\RR^n ; [0, \infty) ) \}.
\end{align*}
Here we recall again that $C_0^\infty(\RR^n;[0 , \infty))$ denotes the space of smooth functions $\varphi:\RR^n\rightarrow [0, \infty)$ with compact support in $\RR^n$. 
Similar to~\cref{lem:firstlem}, we need the following computation.
\begin{Proposition}
\label{prop:1}
Fix $t, \epsilon >0$ and suppose that $s\mapsto G(s)\in C^1([0, t]; M_{n\times n})$ is symmetric for any $s \in [0,t]$ and that \cref{assump:1} is satisfied.  For $f\in  C_\epsilon^\infty(\RR^n; [0, \infty))$ and $s \in (0,t)$, define
\begin{align}
\label{eqn:psidef}
\psi(s):= P_s[(P_{t-s} f) \Gamma^{G(s)} ( \log P_{t-s} f)].
\end{align}
Then
\begin{align*}
\psi^{\prime}(s)&= 2 P_s[(P_{t-s} f) \Gamma_2^{G(s)}(\log P_{t-s} f)]+ P_s [(P_{t-s} f) \Gamma^{G^{\prime}(s)}(\log P_{t-s} f)].
\end{align*}
\end{Proposition}

\begin{proof}
Let $t, \epsilon >0$ and $f\in C_\epsilon^\infty(\RR^n; [0, \infty))$.  Observe that for $s\in (0,t)$
\begin{align}
\nonumber \psi'(s)&=  P_s L[(P_{t-s} f) \Gamma^{G(s)}( \log P_{t-s} f)] -P_s[L(P_{t-s} f) \Gamma^{G(s)} (\log P_{t-s} f)] \\
\label{eqn:gammads1}&\qquad + P_s[(P_{t-s} f) \frac{d}{ds}(\Gamma^{G(s)}(\log P_{t-s} f))].
\end{align}
Next, notice that
\begin{align}
\nonumber \frac{d}{ds}(\Gamma^{G(s)}(\log P_{t-s} f))&= \Gamma^{G'(s)}( \log P_{t-s} f)- 2 \Gamma^{G(s)}( L P_{t-s} f/ P_{t-s}f, \log P_{t-s} f)\\
\label{eqn:gammads2}&=  \Gamma^{G'(s)}( \log P_{t-s} f)- 2 \Gamma^{G(s)}( L( \log P_{t-s} f), \log P_{t-s} f)\\
\nonumber &\qquad - 2 \Gamma^{G(s)}( \Gamma( \log P_{t-s} f), \log P_{t-s} f). \end{align}
Combining~\eqref{eqn:gammads1} with~\eqref{eqn:gammads2} gives
\begin{align}
\nonumber \psi'(s)&= P_s L[(P_{t-s} f) \Gamma^{G(s)}( \log P_{t-s} f)]- P_s[L(P_{t-s} f) \Gamma^{G(s)} (\log P_{t-s} f)] \\
\nonumber &\qquad + P_s[(P_{t-s} f)  \Gamma^{G'(s)}( \log P_{t-s} f) ]\\
\nonumber &\qquad -2P_s[(P_{t-s} f)  \Gamma^{G(s)}( L( \log P_{t-s} f), \log P_{t-s} f) ]\\
\nonumber &\qquad -2P_s[(P_{t-s} f) \Gamma^{G(s)}( \Gamma( \log P_{t-s} f), \log P_{t-s} f) ] \\
 \label{eqn:gammads3}&= 2 P_s[ (P_{t-s} f) \Gamma^{G(s)}_2( \log P_{t-s} f)] +2 P_s[ \Gamma(P_{t-s} f, \Gamma^{G(s)} (\log P_{t-s} f))] \\
\nonumber  &\qquad + P_s[(P_{t-s} f)  \Gamma^{G'(s)}( \log P_{t-s} f) ] -2P_s[(P_{t-s} f) \Gamma^{G(s)}( \Gamma( \log P_{t-s} f), \log P_{t-s} f) ].
 \end{align}
Now observe that if $B=\sigma \sigma^{\ast}/2$, we have
\begin{align*}
\Gamma(P_{t-s} f, \Gamma^{G(s)} (\log P_{t-s} f))&= \sum_{\ell, m} B_{m\ell } \partial_\ell (P_{t-s} f) \partial_m ( G \nabla (\log P_{t-s} f) \cdot \nabla( \log P_{t-s} f)) \\
&= 2 \sum_{\ell, m, j,k} B_{m\ell } \partial_\ell (P_{t-s} f) G_{jk} \partial_{k m}^2 (\log P_{t-s} f) \partial_j ( \log P_{t-s} f)
\end{align*}
and
\begin{align*}
&(P_{t-s} f)\Gamma^{G(s)}(\Gamma( \log P_{t-s} f), \log P_{ t-s} f)\\
 &= (P_{t-s} f)\sum_{j, k} G_{jk} \partial_k (\Gamma( \log P_{t-s} f)) \partial_j( \log P_{t-s} f)\\
& = 2 (P_{t-s} f)\sum_{\ell, m, j, k} G_{jk}   B_{m\ell}  \partial_{km}^2 (\log P_{t-s} f) \partial_\ell (\log P_{t-s} f)) \partial_j( \log P_{t-s} f)\\
&= 2 \sum_{\ell, m, j,k} B_{m\ell } \partial_\ell (P_{t-s} f) G_{jk} \partial_{k m}^2 (\log P_{t-s} f) \partial_j ( \log P_{t-s} f) .
\end{align*}
The result now follows after considering~\eqref{eqn:gammads3} since we just proved that
\begin{align*}
2 P_s[ \Gamma(P_{t-s} f, \Gamma^{G(s)} (\log P_{t-s} f))] = 2P_s[(P_{t-s} f) \Gamma^{G(s)}( \Gamma( \log P_{t-s} f), \log P_{t-s} f) ]. \end{align*}

\end{proof}

As a corollary of the previous computation, we obtain the following result.  Below, we extend $x\mapsto x\log x$ to $[0, \infty)$ by defining it to be zero at $x=0$.

\begin{Corollary}[Reverse log Sobolev inequality]
\label{cor:reverse}
Fix $t>0$, suppose \cref{assump:1} is satisfied and that $s\mapsto G(s,t)$ is as in~\eqref{eqn:Gstdef} . Then we have the following reverse log Sobolev inequality
\begin{align}
\label{eqn:RLS}
(P_t f) \Gamma^{G_t} ( \log P_t f)\leqslant 2 P_t (f \log f) - 2(P_t f) \log P_t f,
\end{align}which is satisfied for all $f\in L^\infty(\RR^n; [0, \infty))$ which are not identically zero.
\end{Corollary}
\begin{proof}
Fix $t, \epsilon >0$.  We first prove the bound for $f\in C_\epsilon^\infty(\RR^n; [0, \infty))$.  First note that if $\psi$ is given by~\eqref{eqn:psidef} with $G(s):=G(s,t)$ where $G(s,t)$ is as in~\eqref{eqn:Gstdef}, by~\cref{prop:1} we have for $s\in (0,t)$
\begin{align*}
\psi'(s)&= 2 P_s[(P_{t-s} f) \Gamma_2^{G(s,t)}(\log P_{t-s} f)]+ P_s [(P_{t-s} f) \Gamma^{\partial_sG(s,t)}(\log P_{t-s} f)] \\
&\geqslant -2 P_s[(P_{t-s} f) \Gamma( \log P_{t-s} f)] = - \frac{d}{ds}  2P_s[(P_{t-s} f)  \log P_{t-s} f].\end{align*}
For $\delta >0$ small, integrating the previous inequality from $\delta$ to $t-\delta$ produces
\begin{align*}
& P_\delta [(P_{t-\delta} f) \Gamma^{G(\delta, t)}(\log P_{t-\delta} f)] -\ P_{t-\delta} [(P_\delta f) \Gamma^{G(t-\delta, t)} (\log P_\delta f)]\\
& \leqslant 2 P_{t-\delta} [(P_\delta f)\log P_\delta f]- 2 P_\delta[(P_{t-\delta} f) \log P_{t-\delta} f].
\end{align*}
Applying a line of reasoning  similar to the one used in the proof of \cref{prop:1}, taking $\delta \rightarrow 0$ gives the desired estimate~\eqref{eqn:RLS} in the case when $f\in C_\epsilon^\infty(\RR^n; [0, \infty))$.  Taking $\epsilon \rightarrow 0$, we obtain the bound for $f\in C_0^\infty(\RR^n; [0, \infty))$ which are not identically zero.  Density arguments give the claimed bound for all $f\in L^\infty(\RR^n; [0, \infty))$ which are not identically zero.
\end{proof}

Next, we turn our attention to establishing a Wang-Harnack inequality.  Later, this will be used to prove a quasi-invariance result when we allow the spatial dimension $n\rightarrow \infty$.  To setup the statement of the result, suppose that $K$ is $n\times n$ symmetric, positive-definite.  We call a curve $\gamma\in C^1( [0, T]; \RR^n)$ \emph{subunit} for $K$ if for all $f\in C^1(\RR^n; \RR)$
\begin{align*}
	\bigg| \frac{d}{ds} f(\gamma(s)) \bigg| \leqslant \sqrt{(\Gamma^Kf)(\gamma(s))}.
\end{align*}
For $x,y \in \RR^n$ and $t>0$, let
\begin{align*}
	S_T^K(x,y) := \{ \gamma:[0, T]\rightarrow \RR^n \text{ subunit for } K  \text{ and } \, \gamma(0)=x, \gamma(T)=y \}
\end{align*}
and define the \emph{control distance} between $x,y \in \RR^n$ associated to $K$: 
\begin{align*}
	\rho^K(x,y) := \inf\{ T>0 \, : \, S_T^K(x,y) \text{ nonempty}\}.
\end{align*}
Note that, in our simplified setting of $\RR^n$, a curve $\gamma\in C^1( [0, T]; \RR^n)$ is subunit for $K$ if and only if for every $t \in [0,T]$, $\|K^{-1/2} \gamma'(t) \| \le 1$ and we therefore have 
\begin{align}
\label{eqn:contdist}
\rho^K(x,y)=\| K^{-1/2} (x-y)\|=\sqrt{\langle x-y, K^{-1}(x-y) \rangle}.
\end{align}

For simplicity, when \cref{assump:1} is satisfied, we use the notation
\begin{align}
\label{def:rho_t}
\rho_t(x,y) := \rho^{G_t}(x,y)
\end{align}
to denote the control distance associated to the matrix $G_t$ as in~\eqref{eqn:G_t}.

%
%

\begin{Theorem}[Wang-Harnack inequality]
\label{thm:wang}
Fix $t>0$, $\alpha >1$ and suppose that \cref{assump:1} is satisfied.  Then we have the estimate
\begin{align}
\label{eqn:WH}
(P_t f (x))^\alpha \leqslant P_t f^\alpha(y) \exp\bigg( \frac{ \alpha \rho_{t}^2(x,y)}{2(\alpha -1)}\bigg).
\end{align}
for all $x,y \in \RR^n$ and all $f\in  L^\infty(\RR^n; [0, \infty))$.
\end{Theorem}
\begin{proof}
We use a similar proof as in \cite[Proposition 3.4]{BaudoinBonnefont2012}. By standard approximation arguments, we may suppose without loss of generality that $f\in C_\epsilon^\infty(\RR^n; [0, \infty))$ for some $\epsilon >0$.  Let $x,y \in \RR^n$.  Since $G_t$ is strictly positive-definite by \cref{lem:pd}, there exists $T>0$ and a curve $\gamma\in S_{T}^{G_t}(x,y)$.  Let $\ell:[0, T]\rightarrow \RR$ be given by $\ell(s)= 1+ (\alpha -1) \tfrac{s}{T}$ and set
\begin{align*}
\psi(s)= \frac{\alpha}{\ell(s)} \log P_t f^{\ell(s)}(\gamma(s)).
\end{align*}
Then applying \cref{cor:reverse} and using subuniticity of $\gamma$, we find that
\begin{align*}
\psi'(s)&= \frac{\alpha(\alpha-1)}{T \ell^2(s)}\frac{P_t f^{\ell(s)} \log f^{\ell(s)} (\gamma(s)) - \log P_t f^{\ell(s)} (\gamma(s)) P_t f^{\ell(s)}(\gamma(s)) }{P_t f^{\ell(s)} (\gamma(s))} \\
& \qquad +\frac{\alpha}{\ell(s)} \frac{\nabla (P_t f^{\ell(s)})(\gamma(s)) \cdot \gamma'(s) }{P_t f^{\ell(s)} (\gamma(s))} \\
& \geqslant \frac{\alpha(\alpha-1)}{ 2T \ell^2(s)} \Gamma^{G_t}( \log P_t f^{\ell(s)}(\gamma(s)))- \frac{\alpha}{\ell(s)} \sqrt{\Gamma^{G_t}( \log P_t f^{\ell(s)}(\gamma(s)) ) }.
\end{align*}
Now for every $\lambda >0$,
\[
-\sqrt{\Gamma^{G_t}( \log P_t f^{\ell(s)}(\gamma(s)) ) }\geqslant -\frac{1}{2\lambda^2}\Gamma^{G_t}( \log P_t f^{\ell(s)}(\gamma(s)) )-\frac{\lambda^2}{2}.
\]
By choosing $\lambda^2=\frac{T \ell (s)}{(\alpha-1)}$ we thus obtain
\begin{align*}
\psi'(s) \geqslant - \frac{\alpha}{2(\alpha-1)}T
\end{align*}
Integrating the inequality above from $0$ to $T$ we find that
\begin{align*}
\alpha \log P_t f(x)- \log P_tf^\alpha(y) \leqslant \frac{\alpha  T^2}{2(\alpha -1)}.
\end{align*}
Exponentiating both sides of the inequality above and optimizing over all such subunit curves $\gamma$ produces the claimed estimate when $f\in C_\epsilon^\infty(\RR^n; [0, \infty))$, finishing the proof.
\end{proof}
%

Using similar computations to those for Corollary \ref{cor:reverse} and Theorem \ref{thm:wang}, we can also obtain a bound on the total variation distance between the measures $P_t(x, \, \cdot \,)$ and $P_t(y, \, \cdot \,)$ for $x,y \in \RR^n$ using the control distance $\rho_t$. 

\begin{Corollary}\label{TV coro}
Let $t>0$ and suppose that \cref{assump:1} is satisfied.  Then we  have
\begin{align*}
\|P_t(x, \, \cdot \, )- P_t(y, \, \cdot \, ) \|_{TV} \leqslant \rho_t(x,y).  
\end{align*}
for all $x,  y \in \RR^n$.
\end{Corollary}

\begin{proof}
Let $f\in  L^\infty(\RR^n)$ with $\|f\|_{L^\infty} \leqslant  1$ and observe that by~\cref{prop:L2smoothing} we  have, for some $\xi \in \RR^n$,
\begin{align*}
|P_t f(x)- P_t f(y) |= | \nabla (P_t f)(\xi) \cdot (x-y) | &= |  G_t^{1/2} \nabla (P_t f)(\xi) \cdot G_{t}^{-1/2}(x-y)| \\
&\leqslant \sqrt{\Gamma^{G_t}(P_t f)(\xi)} \rho_t(x,y) \leqslant \rho_t(x,y),
\end{align*}
where in the final inequality we used~\eqref{eqn:contdist} and \cref{prop:L2smoothing}.    
\end{proof}

\subsection{Transportation cost inequalities}\label{s.TCI}

Corollary \ref{TV coro} can be improved by using the Wasserstein(-Kantorovich-Rubinstein) distance in addition to the total variation distance. Indeed, functional inequalities like in Proposition \ref{prop:L2smoothing} and \ref{cor:reverse} imply transportation cost inequalities as in \cite{BaudoinEldredge2021}.

First recall the definitions of the Hellinger and $2$-Wasserstein distances. Denote by $\mathcal{P}\left( \RR^{n}\right)$ the set of Borel probability measures on $\RR^{n}$ and by $\mathcal{P}^{2}\left( \RR^{n}\right)\subset \mathcal{P}\left( \RR^{n}\right)$ the space of probability measures having a finite second moment.

For $\mu, \nu \in \mathcal{P}^{2}\left( \RR^{n}\right)$ the \emph{$2$-Wasserstein distance} $W_{2}$ is defined by
\[
W_{2}\left( \mu, \nu \right)^2:= \inf \int_{\RR^{n} \times \RR^{n}}
\vert x-y \vert^{2} \pi\left( dx, dy \right),
\]
where the infimum is taken over all couplings $\pi \in \mathcal{P}\left( \RR^{n} \times\RR^{n} \right)$ with marginals $\mu, \nu$.

We will also consider the  \emph{$2$-Wasserstein distance} $W^{\rho_t}_{2}$ for  the control distance associated to the matrix $G_t$ as in~\eqref{eqn:G_t}, which is defined by
\[
W^{\rho_t}_{2}\left( \mu, \nu \right)^2:= \inf \int_{\RR^{n} \times \RR^{n}}
\rho_t(x,y)^{2} \pi\left( dx, dy \right),
\]
where the infimum is taken over all couplings $\pi \in \mathcal{P}\left( \RR^{n} \times\RR^{n} \right)$ with marginals $\mu, \nu$.  Letting  
\begin{align}
\label{def:lam}
\lambda(t) := \inf_{|x|=1} \langle G_t x, x\rangle,
\end{align} 
we observe that \cref{assump:1} implies $\lambda(t)>0$ for all $t>0$.  Furthermore,
Note that
\[
W^{\rho_t}_{2}\left( \mu, \nu \right)^2 \le \frac{1}{ \lambda(t)} W_2(\mu_0, \mu_1)^2,\qquad \mu_0, \mu_1 \in \mathcal{P}(\RR^n).
\]

The \emph{$2$-Hellinger distance} $\mathsf{He}_{2}$ is defined by
\[
\mathsf{He}_{2}^{2}\left( \mu, \nu \right):=
\int_{\RR^{n}} \left( \sqrt{\frac{d \mu}{dm}}-\sqrt{\frac{d \nu}{dm}}\right)^{2}\,dm
\]
where $m$ is any measure that $\mu, \nu$  are both absolutely continuous with respect to (for example, $m=\mu+\nu$). The definition of $\mathsf{He}_{2}$ is independent of $m$. Convergence in the Hellinger distance is equivalent to convergence in total variation, and $\mathsf{He}_{2}^{2}\left( \mu, \nu \right) \leqslant 2$  for all $\mu, \nu \in \mathcal{P}\left( \RR^{n}\right)$, with $\mathsf{He}_{2}^{2}\left( \mu, \nu \right) =2$ if and only if $\mu$ and $\nu$ are mutually singular.

\begin{Proposition}
 Let $t>0$ and suppose that \cref{assump:1} is satisfied.   Then, we have
  \begin{equation*}
    \mathsf{He}_2(\mu_0 P_t, \mu_1 P_t)^2 \leqslant     \frac{1}{ 4} W^{\rho_t}_2(\mu_0, \mu_1)^2,\qquad \mu_0, \mu_1 \in \mathcal{P}(\RR^n).
     \end{equation*}

 \end{Proposition}
\begin{proof}
This follows from Proposition \ref{prop:L2smoothing} and \cite[Theorem 3.7]{BaudoinEldredge2021}.
\end{proof}

Let $\mathrm{Lip}_b(\RR^n)$ denote the space of all bounded Lipschitz functions on $\RR^n$, and for  $a,b \geqslant 0$, let $\mathcal{E}_{a,b}$ denote the class of all
  \emph{positive}   functions $\varphi \in C^1([0,1], \mathrm{Lip}_b(\RR^n))$,
  bounded and bounded away from $0$, satisfying the
  differential inequality
  \begin{equation*}
    \partial_s \varphi_s + a \varphi_s \Gamma^{G(t)}( \ln \varphi_s)  + b
    \varphi_s \ln \varphi_s \leqslant 0.
  \end{equation*}
For probability measures $\mu_0, \mu_1 \in \mathcal{P}(\RR^n)$, we define
  \begin{equation*}
    T_{a,b}(\mu_0, \mu_1) := \sup_{\varphi \in \mathcal{E}_{a,b}}
      \left[\int_{\RR^n} \varphi_1\,d\mu_1 - \int_{\RR^n} \varphi_0\,d\mu_0\right].
  \end{equation*}
 Provided $b>0$, note by \cite[Proposition 5.11]{BaudoinEldredge2021} we have
  \begin{equation*}
    T_{0,b}(\mu_0, \mu_1) = \begin{cases} C_b \int_{\RR^n}
      \left(\frac{d\mu_1}{d \mu_0}\right)^{\frac{e^b}{e^b-1}}\,d\mu_0, & \mu_1 \ll
      \mu_0 \\
\infty, & \mu_1 \not\ll \mu_0
\end{cases}
  \end{equation*}
  where $C_b=\frac{e^b-1}{e^{be^b}}$. 

\begin{Theorem}
\label{thm:transport}
Let $t>0$ and suppose that \cref{assump:1} is satisfied. Then
 \begin{equation}\label{eti}
 T_{0, \kappa / \lambda(t)}(\mu_0 P_t, \mu_1 P_t) \leqslant T_{\kappa, \kappa / \lambda(t)}(\mu_0,
        \mu_1), \qquad \mu_0, \mu_1 \in \mathcal{P}(\RR^n), \quad \kappa >0.
      \end{equation}
In particular, for every $x,y \in \RR^n$ and $t>0$
 \begin{equation*}
      \int_{\RR^n} \left(\frac{p_t(x,z)}{p_t(y,z)}\right)^{1/(p-1)} p_t (x,z)\, dz  \leqslant       \exp\left(\frac{p}{(p-1)^2} \frac{ \rho_t(x,y)^2}{2}\right), \qquad
      p > 1.
    \end{equation*}
\end{Theorem}

\begin{proof}
	The entropic transportation inequality \eqref{eti} follows from Corollary \ref{cor:reverse} and  \cite[Theorem 5.15]{BaudoinEldredge2021} since we have Wang-Harnack inequality \eqref{eqn:WH} for any $p >1$
\[
(P_t f (x))^{p} \leqslant C^{p} P_t f^{p}(y) ,
\]
with
	\[ C:=\exp\left( \frac{ \rho_{t}^2(x,y)}{2(p -1)}\right).
	\]
As was observed in different settings in \cite[Lemma D1]{DriverGordina2009}, \cite[Lemma 2.11]{BaudoinGordinaMelcher2013} and \cite[Proposition 4.1]{Gordina2017}, this is equivalent to the integrated Harnack inequality with $p^{\prime}=\frac{p}{p -1}$:
\begin{align*}
\int_{\RR^n} \left(\frac{p_t(x,z)}{p_t(y,z)}\right)^{p^{\prime}} p_t (x,z)\, dz  
	\leqslant C^{p^{\prime}}&=\exp\left( \frac{ p^{\prime} \rho_{t}^2(x,y)}{2(p -1)}\right)
 =\exp\left(\frac{p}{(p-1)^2} \frac{ \rho_{t}^2(x,y)}{2}\right).
\end{align*}

\end{proof}

 \begin{Remark}
 We now observe that \cref{thm:transport} implies a quasi invariance result (see also~\cite{Gordina2017}).  Relying on \cite[Lemma 5.10]{BaudoinEldredge2021} which says that if for $b>0$ we have $T_{0, b}(\mu_{0}, \mu_{1}) < \infty$, then $\mu_{0}$ is absolutely continuous with respect to $\mu_{1}$. Taking in~\eqref{eti} with $\mu_{0}=\delta_{x}$ and  $\mu_{1}=\delta_{y}$, then $T_{0, \kappa / \lambda(t)}(\delta_{x} P_t, \delta_{y} P_t) < \infty$ and therefore $\delta_{x} P_t$ is absolutely continuous with respect to  $\delta_{y} P_t$ and by symmetry $\delta_{y} P_t$ is absolutely continuous with respect to  $\delta_{x} P_t$.
\end{Remark}

\subsection{The iterate $G_2$ of $G$}
Before proceeding onto concrete applications of the matrix $G(s,t)$ as in~\eqref{eqn:Gstdef}, in this section, we discuss the \emph{iterate} $G_2$ of this matrix.  From a practical standpoint, the iterate $G_2$ allows further stochastic mixing while retaining some structure similar to $G_t$. In particular in some applications, it will allows us to more easily control the smallest positive eigenvalue $\lambda(t)$ of the matrix $G_t$ introduced in~\eqref{def:lam}.

Fixing $t>0$, we define the \emph{iterate} $s\mapsto G_2(s,t)$ of the matrix~\eqref{eqn:Gstdef} on the interval $[0,t]$ by the formula
\begin{align*}
G_2(s,t):= \int_s^t e^{(s-v)A} G(v,t)e^{(s-v) A^{\ast}} \, dv.
\end{align*}
Observe that $s\mapsto G_2(s,t)$ is the unique solution of the backwards ODE
\begin{align*}
\begin{cases}
\partial_s G_2(s,t) = - G(s,t) + A G_2(s,t) + G_2(s,t) A^{\ast}\\
G_2(t,t)=0.
\end{cases}
\end{align*}
Although we could re-do the Bakry-\'{E}mery calculus used in this section with $G_2(s,t)$ in place of $G(s,t)$ to produce similar functional inequalities, we find it more expedient to compare their respective spectra.  In particular, our main result in this section is~\cref{prop:G_2} below.

To this end, for $t>0$, let $\lambda_{2}(t):= \inf_{|x|=1} \langle G_2(0,t) x,x \rangle$ be the smallest eigenvalue of the matrix $G_{2,t}:=G_2(0,t)$.  We have the following:

\begin{Proposition}
\label{prop:G_2}
For $t>s>0$, we have the bound
\begin{align}
\label{eqn:evbound}
\lambda(t) \geqslant \frac{\lambda_{2}(t)}{t}\geqslant \frac{\lambda(t-s)}{t}  \inf_{|x|=1} \int_0^{s} |e^{-vA^{\ast}} x|^2 \, dv.
\end{align}
\end{Proposition}

\begin{proof}
Fix $t>0$ and $x\in \RR^n$.  We first prove the inequality $\lambda(t) \geqslant \lambda_{2}(t)/t$.  Observe that
\begin{align*}
\langle G_{2,t} x,x \rangle = \int_0^t \langle G(v,t) e^{-v A^{\ast}} x, e^{-v A^{\ast} }x\rangle\, dv &= \int_0^t \int_v^t | \sigma^{\ast} e^{(v-w)A^{\ast}} e^{-vA^{\ast}} x|^2 \, dw dv\\
&\leqslant \int_0^t \int_0^t | \sigma^{\ast} e^{-wA^{\ast}} x|^2 \, dw \, dv= t \langle G_t x,x \rangle.
\end{align*}
Taking the infimum over $x\in \RR^n$ with $|x|=1$ implies $\lambda(t) \geqslant \lambda_{2}(t)/t$.

For the remaining bound, fix $s>0$ with $s<t$.  For $x\in \RR^n$, since $G(v,t)$ is nonnegative, we have that
\begin{align}
\nonumber \langle G_{2,t} x,x \rangle &= \int_0^t  \langle G(v,t) e^{-v A^{\ast}} x, e^{-v A^{\ast}} x\rangle\, dv \\
\label{eqn:Glb1}&\geqslant \int_0^{s}  \langle G(v,t) e^{-v A^{\ast}} x, e^{-v A^{\ast} }x\rangle\, dv.\end{align}
Now, for $v\in [0, s]$ we have for any $x\in \RR^n$
\begin{align}
\label{eqn:Glb2}
 \langle G(v,t)  x, x\rangle = \int_v^t |\sigma^{\ast} e^{(v-w)A^{\ast}} x|^2 \, dw= \int_0^{t-v} |\sigma^{\ast} e^{-w A^{\ast}} x|^2 \, dw\geqslant \lambda(t-s) |x|^2.
 \end{align}
 Combining \eqref{eqn:Glb1} with \eqref{eqn:Glb2} produces
 \begin{align*}
  \langle G_{2,t} x,x \rangle \geqslant \lambda(t-s) \int_0^{s}  |e^{-v A^{\ast}} x|^2 \, dv,
   \end{align*}
   which implies the remaining inequality.
\end{proof}

\begin{Remark}
The bound~\eqref{eqn:evbound} allows to utilize the underlying dynamics driven by the ODE $\dot{x}= - A^{\ast} x$ to lower bound $\lambda(t)$.  In practice, dissipative dynamics for the system $\dot{x}=A^{\ast} x$ leads to explosive dynamics for the time-reversed system $\dot{x}= -A^{\ast} x$.  In such cases, it is often easy to show explicitly that
\begin{align*}
\int_0^{s} |e^{-v A^{\ast}} x|^2 \, dv
\end{align*}
grows exponentially fast as $s\rightarrow \infty$.  We refer the reader to \cref{sec:oscdamp} for a concrete example.
\end{Remark}

\section{Estimating the distance $\rho_t$ in examples}
\label{sec:class}

The goal of this section is to bound the control distance $\rho_t$ as in~\eqref{def:rho_t} from above in a number of conrete examples in the form~\eqref{eqn:SDE} where the noise is degenerate but \cref{assump:1} is satisfied.  Of particular importance will be to produce explicit estimates depending on key parameters in the specific system. Looking ahead to the following section, how the estimates depend on the spatial dimension parameter will be critical as we consider an infinite-dimensional class of examples related to the examples treated in this section.

First, in~\cref{sub:genstruct}, we outline the general form of each example and use it to deduce a basic uppers bounds for $\rho_t$ in terms of the \emph{underlying matrices} defined below.  Importantly, the matrices $A$ and $\sigma$ will be written in tensored form from which some simplifications can be deduced.  Then, in~\cref{sub:ex}, we consider four, specific examples where  we  estimate $\rho_t$ using the underlying matrices via \cref{prop:tensor} below.

\subsection{Kronecker product form of the examples}
\label{sub:genstruct}
Letting $j,k\in \N$, $I=I_{k\times k}$ denote the $k\times k$ identity matrix and $Q$ denote a $k\times k$, symmetric strictly positive-definite matrix, throughout this section $A$ and $\sigma$ are $n\times n$ matrices, $n=jk$, of the following form
\begin{align}
\label{eqn:exAsig}
A=  \begin{bmatrix}
a_{11} I & a_{12} I &  \ldots & a_{1j} I\\
a_{21} I & a_{22} I & \ldots & a_{2j} I \\
\vdots & \vdots & \ldots & \vdots \\
a_{j1} I & a_{j2} I & \ldots & a_{jj} I
\end{bmatrix} \,\,\, \,\,\text{ and } \,\,\,\, \, \sigma=  \begin{bmatrix}
\sigma_{11} Q^{1/2} & \sigma_{12} Q^{1/2} &  \ldots & \sigma_{1j} Q^{1/2}\\
\sigma_{21} Q^{1/2} & \sigma_{22} Q^{1/2} & \ldots & \sigma_{2j} Q^{1/2} \\
\vdots & \vdots & \ldots & \vdots \\
\sigma_{j1} Q^{1/2} & \sigma_{j2} Q^{1/2} & \ldots & \sigma_{jj} Q^{1/2}
\end{bmatrix}.\end{align}
To $A$ and $\sigma$, we associate $j\times j$ real matrices $\underline{A}$ and $\underline{\sigma}$, called the \emph{underlying matrices} corresponding to $A$ and $\sigma$, given by
  \begin{align*}
\underline{A}=  \begin{bmatrix}
a_{11} & a_{12}  &  \ldots & a_{1j} \\
a_{21}  & a_{22}  & \ldots & a_{2j}  \\
\vdots & \vdots & \ldots & \vdots \\
a_{j1} & a_{j2}  & \ldots & a_{jj}
\end{bmatrix} \qquad \text{and} \qquad \underline{\sigma}=  \begin{bmatrix}
\sigma_{11}  & \sigma_{12}  &  \ldots & \sigma_{1j} \\
\sigma_{21}  & \sigma_{22}  & \ldots & \sigma_{2j}  \\
\vdots & \vdots & \ldots & \vdots \\
\sigma_{j1}  & \sigma_{j2}  & \ldots & \sigma_{jj}
\end{bmatrix}.\end{align*}
The parameter $j$ will be called the \emph{underlying dimension}.  Observe that we can write $A$ and $\sigma$ in terms of the underlying matrices $\underline{A}$ and $\underline{\sigma}$ using the Kronecker product as
\begin{align*}
A= \underline{A}\otimes I  \qquad \text{ and } \qquad \sigma= \underline{\sigma} \otimes Q^{1/2}.
\end{align*}
Consequently, using symmetry of $Q$, we can write the corresponding matrix $G_t$ as in~\eqref{eqn:G_t} as
\begin{align}
 G_t\nonumber=  \int_0^t e^{-v A} \sigma \sigma^{\ast}  e^{-vA^{\ast}} \, dv &= \int_0^t (e^{-v \underline{A}} \otimes I) (\underline{\sigma}\otimes Q^{1/2}) (\underline{\sigma}\otimes Q^{1/2})^{\ast} (e^{-v \underline{A}^{\ast}} \otimes I) \, dv  \\
 &= \int_0^t e^{-v \underline{A}} \underline{\sigma} \, \underline{\sigma}^{\ast} e^{-v \underline{A}^{\ast}} \, dv \otimes Q \nonumber \\
\label{eqn:tensorformulaG}
	&=: \underline{G}_t \otimes Q.
\end{align}

Notationally, throughout this section, for any $x\in \RR^n$ and any $k\times k$ matrix $B$ we write 
\begin{align}
x= \begin{bmatrix}
x_1 \\
x_2\\
\vdots\\
x_j 
\end{bmatrix}\in (\RR^k)^j
\qquad \text{ and } \qquad 
\hat{B} x = \begin{bmatrix}
B x_1 \\
Bx_2\\
\vdots\\
B x_j
\end{bmatrix} \in (\RR^k)^j
\end{align}  
where the $x_\ell$'s above are $k$-dimensional column vectors. 
Setting
\begin{align}
\label{eqn:kalsub}
\underline{A}_{\underline{\sigma}}&:=\begin{bmatrix}
\underline{\sigma} & \underline{A}\, \underline{\sigma} & \ldots & \underline{A}^{j-1} \underline{\sigma}
\end{bmatrix},\\
\label{def:lambda}
\underline{\lambda}(j,t)&:= \inf_{|x|_{\RR^j}=1} \langle \underline{G}_t x, x\rangle_{\RR^j},
\end{align}
we obtain the following result.
\begin{Proposition}
\label{prop:tensor}
Suppose that $\underline{A}_{\underline{\sigma}}$ has full rank, $x,y \in \RR^n$ and $t>0$.  Then $\underline{G}_t$ is invertible, $\underline{\lambda}(j,t)>0$ and 
\begin{align}
\label{eqn:rho_tb1}
\rho_t(x,y)^2  &= \rho^{\underline{G}_t \otimes I}(\widehat{Q^{-1/2}}x, \widehat{Q^{-1/2}}y)^2 \leq \sum_{\ell=1}^j \frac{ \langle Q^{-1} (x_\ell - y_\ell), x_\ell-y_\ell \rangle_{\RR^k}}{\underline{\lambda}(j,t)}.
\end{align}
\end{Proposition}

\begin{proof}
The equality in~\eqref{eqn:rho_tb1} follows from~\eqref{eqn:contdist} and properties of the Kronecker product.  The estimate in~\eqref{eqn:rho_tb1} follows by definition of $\underline{\lambda} (j,t)$ and the fact that if $v=(v^1, \ldots, v^j)\in \RR^j$ is an eigenvector of a $j\times j$ matrix $C$ with eigenvalue $\lambda$, then for every $i=1,2,\ldots, k$ the vector
\begin{align*}
\begin{bmatrix}
v^1 e_i \\v^2 e_i \\\vdots \\ v^j e_i
\end{bmatrix},
\end{align*}
with $e_i $ denoting the standard orthonormal basis element of $\RR^k$, is an eigenvector with eigenvalue $\lambda$ for the matrix $C\otimes I$, $I=I_{k\times k}$. 
\end{proof}


\subsection{Examples}
\label{sub:ex}
We next consider a number of examples with $A$ and $\sigma$ in the form~\eqref{eqn:exAsig} in which we estimate the control distance $\rho_t$ as in~\eqref{def:rho_t} using \cref{prop:tensor} and the structure afforded in the specific dynamics.

\begin{Example}[Kolmogorov diffusion]\label{Ex.1}
Consider the Kolmogorov diffusion, whose SDE is of the form \eqref{eqn:SDE} with $n=2k$ for some $k\in \N$, and
\begin{align*}
A= \begin{bmatrix} 0 & 0 \\ I & 0 \end{bmatrix}
\qquad\text{ and } \qquad \sigma=
\begin{bmatrix} Q^{1/2} & 0\\ 0 & 0 \end{bmatrix}
\end{align*}
where $I = I_{k\times k}$ and $Q$ is $k\times k$, symmetric and strictly positive definite. In this case, the underlying dimension is $j=2$ and the associated underlying matrices are given by \begin{align*}
\underline{A}=
\begin{bmatrix} 0 & 0\\ 1 & 0 \end{bmatrix} \qquad  \text{ and } \qquad
\underline{\sigma}=
\begin{bmatrix} 1 & 0 \\ 0 & 0 \end{bmatrix}.
\end{align*}
Observe that
\begin{align*}
\underline{A}_{\underline{\sigma}} = \begin{bmatrix}
\underline{\sigma} & \underline{A} \underline{\sigma}
\end{bmatrix}= \begin{bmatrix}
1&0&0&0 \\
0&0&1&0 \end{bmatrix}
\end{align*}
has full rank.  Furthermore, $\underline{A}$ is nilpotent with $\underline{A}^2 =0$.  Hence 
\begin{align*}
e^{-v \underline{A}} \sigma = (I_{2\times 2}-v \underline{A})\sigma =
\begin{bmatrix}
1 & 0 \\ -v & 0 \end{bmatrix},
\end{align*}
so that
\begin{align*}
	\underline{G}_t =\int_0^t e^{-v \underline{A}} \underline{\sigma} \,\underline{\sigma}^{\ast} e^{-v \underline{A}^{\ast}} \, dv = \int_0^t \begin{bmatrix} 1 & -v \\ - v & v^2 \end{bmatrix} dv = \begin{bmatrix} t & -\frac{t^2}{2} \\[3pt] -\frac{t^2}{2} & \frac{t^3}{3} \end{bmatrix}.
\end{align*}
We observe that the matrix $\underline{G}_t$ has inverse given by 
\begin{align*}
\underline{G}_t^{-1}= \begin{bmatrix}
\frac{4}{t} & \frac{6}{t^2}\\
\frac{6}{t^2}& \frac{12}{t^3}
\end{bmatrix},
\end{align*}
so that by \cref{prop:tensor}, if $\hat{x}= \widehat{Q^{-1/2} }x$ and $\hat{y}=\widehat{Q^{-1/2} }x$, we have 
\begin{align*}
\rho_t(x,y)^2 = \rho^{\underline{G}_t\otimes I}( \hat{x}, \hat{y})^2&= \langle (\underline{G}_t^{-1} \otimes I) (\hat{x}- \hat{y}), \hat{x}- \hat{y} \rangle \\
 &= \frac{4}{t}\| Q^{-1/2} (x_1-y_1) \|^2_{\RR^k} + \frac{12}{t^2}\langle Q^{-1} (x_1-y_1), x_2-y_2 \rangle_{\RR^k}\\
 &\qquad +\frac{12}{t^3}\| Q^{-1/2}(x_2-y_2)\|^2_{\RR^k} . 
 \end{align*}
Furthermore, we can calculate the smallest positive eigenvalue $\underline{\lambda}(2, t)>0$ of $\underline{G}_{t}$ to see that 
\begin{align*}
	\underline{\lambda}(2,t)= \frac{t+ \tfrac{t^3}{3} - \sqrt{(t+t^3/3)^2-t^4/3}}{2} &= \frac{t^4}{6(t+ \tfrac{t^3}{3} + \sqrt{(t+t^3/3)^2-t^4/3} )}\\
&\simeq \begin{cases}
\frac{t}{4} & \text{ as } t\rightarrow \infty \\
\frac{t^3}{12} & \text{ as } t\rightarrow 0^+.
\end{cases}
\end{align*}

Observe that the expression above for the distance $\rho_t$ also reveals a scale invariance for this particular example.  That is, if we define for each $a>0$ a dilation $\delta_a:\mathbf{R}^{2k}\to\mathbf{R}^{2k}$ by
	\begin{align*}
	\delta_{a}x=\delta_a \begin{bmatrix}
	x_1 \\
	x_2
	\end{bmatrix}
			= \begin{bmatrix} ax_1 \\ a^3x_2\end{bmatrix}, 
			\end{align*}
then we see that
	\[ \rho_t(x,y) = \rho_1(\delta_{t^{-1/2}}x,\delta_{t^{-1/2}}y) .\]
This is appropriate for this example as it coincides with the natural scale invariance that the solution $x_t$ inherits from the standard scaling relation for Brownian motion 
\begin{align*}
\sqrt{a}B_t^x\overset{d}{=} B_{at}^{\sqrt{a}x},
\end{align*}
where $B_t^x := x+ B_t $. To see this explicitly, note that for the given $A$ and $\sigma$, we may express the solution of \eqref{eqn:gaussian} as 
\[ x_t(x) = \begin{bmatrix} x_1 + Q^{-1/2} B_t^1 \\ x_2 + tx_1 + \int_0^t Q^{-1/2} B_s^1\,ds \end{bmatrix},\]
	where $\{B_t^1\}_{t\ge0}$ is a standard $k$-dimensional Brownian motion.  From this expression it follows that 
		$\{\delta_{a}x_t(x)\}\overset{d}{=} \{x_{a^2t}(\delta_a(x))\}$.
\end{Example}

Note that control distances in more general settings are not so easily computable.  In particular, the scale invariance observed in Example \ref{Ex.1} will not hold for general $A$ and $\sigma$. However, there are other examples where such relations are possible.

\begin{Example}[Iterated Kolmogorov diffusion]\label{Ex.2}
Consider next the SDE~\eqref{eqn:SDE} with $n=jk$ and $n\times n$ matrices $A$ and $\sigma$ given by
\begin{align*} 
	A =
\begin{bmatrix}
0 & 0 & \cdots & 0 & 0 \\
I & 0 & \cdots & 0 & 0 \\
0 & I & \cdots & 0 & 0\\
\vdots & \vdots & \ddots & \vdots & \vdots\\
0 & 0 & \cdots & I & 0\\
\end{bmatrix}
\ \ \ \ \mbox{ and } \ \ \
\sigma =
\begin{bmatrix}
I & 0 & \cdots & 0\\
0 & 0 & \cdots & 0\\
\vdots & \vdots & \ddots & \vdots\\
0 & 0 & \cdots & 0
\end{bmatrix},
\end{align*}
	with $I=I_{k\times k}$.  We observe that, in this case, equation~\eqref{eqn:SDE} has explicit solution given by 
	\[ x_t(x) =x+ \left( B_t, \int_0^t B_s\,ds, \ldots, \int_{\Delta_{j-1}(t)} B_{s_{j-1}} \,ds_{j-1} \cdots ds_1 \right) \]
	with $\{B_t\}_{t\ge0}$ a standard $k$-dimensional Brownian motion and $\Delta_\ell(t)=\{0<s_\ell<s_{\ell-1}<\cdots<s_1<t\}$. 
	As with the Kolmogorov diffusion, it is also possible to explicitly compute $G_t$, but one may also see the scale invariance just from the above expression for $x_t$.
	If we define $\delta_a:\RR^{jk}\to\RR^{jk}$ by
\[ \delta_ax = \delta_a(x_1,\ldots, x_j) := \left(ax_1,\ldots, a^{2j-1}x_j \right) \]
then the standard scaling relation for $\{B_t\}$ implies that
	\[ \{\delta_{a} x_t(0)\} \overset{d}{=} \{x_{a^2t}(0)\} \]
and, similarly, 
	\[ \{\delta_{a} (x_t(x))\} \overset{d}{=} \{x_{a^2t}(\delta_{a}(x)) \}.\]
This can be expressed in semigroup form as
\[ P_t(f\circ \delta_a) = (P_{a^2t}f) \circ \delta_a.\]
In particular, taking $f=f\circ\delta_a$ in the Wang-Harnack inequality, cf. \cref{thm:wang}, yields
	\begin{align*}
		((P_t(f\circ\delta_a))(x))^\alpha 
		&= (P_{a^2t}f)^\alpha(\delta_a(x)) \\
			&\leqslant \exp\left(\frac{\alpha}{2(\alpha-1)}\rho_{a^2t}(\delta_ax,\delta_ay) \right) (P_{a^2t}(f^\alpha))(\delta_a(y)) \\
			&= \exp\left(\frac{\alpha}{2(\alpha-1)}\rho_{a^2t}(\delta_ax,\delta_ay) \right) (P_{t}((f\circ \delta_a)^\alpha))(y).
	\end{align*}
We may now take $f=f\circ \delta_{a^{-1}}$ which gives
	\[ ((P_tf)(x))^\alpha \leqslant \exp\left(\frac{\alpha}{2(\alpha-1)}\rho_{a^2t}(\delta_ax,\delta_ay) \right) (P_{t}(f^\alpha))(y). \]
Note that the choice of $a=t^{-1/2}$ gives
	\[ ((P_tf)(x))^\alpha \leqslant \exp\left(\frac{\alpha}{2(\alpha-1)}\rho_{1}(\delta_{t^{-1/2}}x,\delta_{t^{-1/2}}y) \right) (P_{t}(f^\alpha))(y). \]
\end{Example}

\subsubsection{Linear kinetic Fokker-Planck equation}
\label{sec:KFP}
Recalling that $I=I_{k\times k}$ is the identity and $Q$ is $k\times k$ and positive-definite, consider the case when $A$ and $\sigma$ are $2k\times 2k$ of the following forms
\begin{align*}
A= \begin{bmatrix}
0& I \\ - I & - \gamma I
\end{bmatrix} \qquad\text{ and } \qquad \sigma= \begin{bmatrix}
0&0 \\ 0 & \sqrt{\gamma} Q^{1/2}
\end{bmatrix}
\end{align*}
where $\gamma >0$ is a constant, called the \emph{friction} parameter, and the matrix $Q$ is independent of $\gamma$.
In this case, the underlying matrices $\underline{A}$ and $\underline{\sigma}$ are $2\times 2$ and satisfy
\begin{align*}
\underline{A}= \begin{bmatrix}
0& 1 \\ -1 & -\gamma
\end{bmatrix} \qquad\text{ and } \qquad \underline{\sigma}= \begin{bmatrix}
0 & 0 \\ 0 & \sqrt{\gamma}
\end{bmatrix}.\end{align*}
Note that
\begin{align}
\label{eqn:KRC}
\underline{A}_{\underline{\sigma}} = \begin{bmatrix}
\underline{\sigma} & \underline{A} \underline{\sigma}
\end{bmatrix}= \begin{bmatrix}
0&0&0&\sqrt{\gamma} \\
0&\sqrt{\gamma}&0&-\gamma^{3/2}
\end{bmatrix}
\end{align}
has full rank so that $\underline{\lambda}(2,t)>0$ for all $t>0$.  Our goal will be to study $\underline{\lambda}(2,t)$ in the regimes where $\gamma \approx 0$ and $\gamma \gg 1$.  This analysis will be done on the appropriate time scale depending on the regime.  That is, when $\gamma \approx 0$, we set $t=t_*/\gamma$ and when $\gamma \gg 1$, we let $t=t_* \gamma$ where $t_*\geqslant 1$.  These timescales correspond to the correct scaling of the mixing rate of the Markovian dynamics with respect to $\gamma$~\cite{CamrudGordinaHerzogStoltz2022}.  Below, we assume that $\gamma \neq 2$.

We find it convenient to diagonalize $\underline{A}$.  First, observe that $\underline{A}$ has distinct eigenvalues $\lambda_1, \lambda_2 \in \mathbf{C}$ given by
\begin{align}
\label{eqn:lamb}
\lambda_1 =  \frac{-\gamma+ \sqrt{\gamma^2-4}}{2}\qquad \text{ and } \qquad \lambda_2 = \frac{-\gamma- \sqrt{\gamma^2-4}}{2}
\end{align}
where $\sqrt{-1}=i$.  Furthermore, we can write
\begin{align*}
\underline{A}= \begin{bmatrix} 1&1\\
\lambda_1 & \lambda_2  \end{bmatrix}\begin{bmatrix} \lambda_1&0\\
0 & \lambda_2  \end{bmatrix}\begin{bmatrix} 1&1\\
\lambda_1 & \lambda_2  \end{bmatrix}^{-1}.\end{align*}
Hence,
\begin{align*}
e^{-\underline{A} t} \underline{\sigma} =- \frac{\sqrt{\gamma}}{\sqrt{\gamma^2-4}} \begin{bmatrix} 0& e^{-\lambda_2 t} - e^{-\lambda_1 t}\\
0& \lambda_2 e^{-\lambda_2 t} - \lambda_1 e^{-\lambda_1 t}\end{bmatrix}.
\end{align*}
Thus, if $c_\gamma = \gamma/(\gamma^2-4)$, we have
\begin{align*}
&\underline{G}_t 
= \int_0^t e^{-\underline{A} s} \underline{\sigma} \underline{\sigma}^{\ast} e^{-\underline{A}^{\ast} s} \, ds\\
&=c_\gamma \int_0^t\begin{bmatrix}
e^{-2\lambda_1 s} -2 e^{-(\lambda_1+ \lambda_2) s} + e^{-2\lambda_2 s} & \lambda_1 e^{-2\lambda_1 s} - (\lambda_1+ \lambda_2) e^{-(\lambda_1+ \lambda_2) s} + \lambda_2 e^{-2\lambda_2 s} \\\
\lambda_1 e^{-2\lambda_1 s} - (\lambda_1+ \lambda_2) e^{-(\lambda_1+ \lambda_2) s} + \lambda_2 e^{-2\lambda_2 s}& \lambda_1^2 e^{-2\lambda_1 s} - 2\lambda_{1} \lambda_2 e^{-(\lambda_1 + \lambda_2) s} + \lambda_2^2 e^{-2\lambda_2 s}
\end{bmatrix}\, ds\\
&=: c_\gamma \begin{bmatrix}
G_{11} & G_{12} \\
G_{12} & G_{22}
\end{bmatrix}.
\end{align*}
Observe that for all $\gamma \neq 2$, by positive-definiteness of $\underline{G}_t$ via~\eqref{eqn:KRC}, the discriminant
\begin{align*}
D:= G_{11} G_{22} - G_{12}^2
\end{align*}
is strictly positive for all $t>0$.

\vspace{0.1in}

\noindent \underline{\emph{Case 1} ($\gamma \gg 1$)}.  In this case, $\lambda_1\neq \lambda_2 $ are both real and we set $t= \gamma t_*$ where $t_* \geqslant 1$ is independent of $\gamma$.  Moreover, since $c_\gamma >0$,
\begin{align}
\nonumber
 \underline{\lambda}(j,t)&=c_\gamma \frac{ G_{11}+ G_{22} - \sqrt{(G_{11}+ G_{22})^2 - 4D}}{2}=c_\gamma \frac{2D}{G_{11}+ G_{22} + \sqrt{(G_{11}+G_{22})^2 -4D}}.
\end{align}
Using the explicit expressions above along with the values of $\lambda_1, \lambda_2$ in~\eqref{eqn:lamb}, we obtain
\begin{align}
\label{eqn:Ggamlarge}
	G_{11} &= \frac{2}{\gamma}- \frac{\gamma}{2}  - \frac{2}{\gamma}e^{\gamma t} - \frac{\lambda_2}{2} e^{-2\lambda_1 t} - \frac{\lambda_1}{2}e^{-2\lambda_2 t}, \\
 G_{22}&= \frac{2}{\gamma}- \frac{\gamma}{2} - \frac{2}{\gamma} e^{\gamma t} - \frac{\lambda_1}{2}e^{-2\lambda_1 t} - \nonumber 
	\frac{\lambda_2}{2}e^{-2\lambda_2 t },\\
\nonumber 
	G_{12}&= e^{\gamma t} - \frac{e^{-2\lambda_1 t}}{2} - \frac{e^{-2\lambda_2 t}}{2}.
\end{align}
Furthermore, note that as $\gamma \rightarrow \infty$ we have the following asymptotic formulas
\begin{align}
\label{eqn:lambdaa1}
\lambda_1&= - \frac{1}{\gamma} + O(\gamma^{-2})\qquad \text{ and } \qquad  \lambda_2 = -\gamma + \frac{1}{\gamma} + O(\gamma^{-2}) .
\end{align}
Substituting in $t= t_* \gamma$ into the expressions following~\eqref{eqn:Ggamlarge} and using~\eqref{eqn:lambdaa1} produces the following bound
\begin{align*}
	D&= G_{11}G_{22}- G_{12}^2 \\
	&= \bigg\{\frac{e^{-4\lambda_2 \gamma t_*}}{4} - e^{(-2\lambda_2+\gamma)\gamma t_*} + \frac{\lambda_2^2}{4} e^{2\gamma^2 t_*} + \frac{\lambda_2 \gamma}{4} e^{-2\lambda_2 \gamma t_*}+ R_1\bigg\}\\
	& \qquad - \bigg\{\frac{e^{-4\lambda_2 \gamma t_*}}{4} - e^{(-2\lambda_2 + \gamma) \gamma t_*} + \frac{3}{2}e^{2\gamma^2 t_*} + R_2  \bigg\}\\
&\geqslant \frac{\lambda_2^2}{4} e^{2\gamma^2 t_*} \bigg \{ 1 - e^{-t_*}  -C\gamma^{-2}\bigg\}
\end{align*}
which is satisfied for all $t_*\geqslant 1$ and $\gamma \geqslant \gamma_\ell \gg 1$ for some constant $C>0$ independent of $t_*, \gamma$.
In a similar fashion, by increasing $\gamma_\ell>0$ and $C>0$ if needed, we obtain\begin{align*}
G_{11}+ G_{22} + \sqrt{(G_{11}+G_{22})^2 - 4D} &\leqslant  -\frac{ 3\lambda_2}{2} e^{-2\lambda_2 \gamma t_*} \bigg \{ 1 + C/\gamma \bigg\}
\end{align*}
for all $t_* \geqslant 1, \gamma \geqslant  \gamma_\ell$.
Consequently, by increasing $\gamma_\ell$ again if needed, we see that  for $t_* \geqslant 1$ and $\gamma \geqslant \gamma_L$
\begin{align*}
\lambda(2, t_* \gamma) \geqslant \frac{-c_\gamma \lambda_2}{6} e^{t_*} \frac{1-e^{-t_*} - C\gamma^{-2}}{1+C\gamma^{-1}} \geqslant \frac{e^{t_*}}{16}.
\end{align*}

\vspace{0.1in}

\noindent \underline{\emph{Case II} ($\gamma \approx 0$)}.  In this case, we set $t= t_*/\gamma$ where $t_* \geqslant 1$.  In this case, since $c_\gamma <0$,
\begin{align*}
\underline{\lambda}(2, t) = c_\gamma \frac{ (G_{11}+ G_{22}) + \sqrt{(G_{11}+ G_{22})^2 - 4D}}{2}=|c_\gamma| \frac{2D}{-(G_{11}+ G_{22}) + \sqrt{(G_{11}+G_{22})^2 -4D}}.\end{align*}
Again, since $\gamma \approx 0$, we find that if $\beta = |\sqrt{\gamma^2 -4}|/2$, then
\begin{align*}
G_{11}&= - \frac{\gamma}{2}+ \frac{2}{\gamma}[1- e^{\gamma t}] + \frac{\gamma}{2}e^{\gamma t} \cos( 2\beta t) + \beta e^{\gamma t} \sin (2\beta t),\\
	G_{22}&=- \frac{\gamma}{2} + \frac{2}{\gamma}[ 1- e^{\gamma t}] + \frac{\gamma}{2}e^{\gamma t} \cos (2\beta t) - \beta e^{\gamma t} \sin(2\beta t)\\
	G_{12}&=e^{\gamma t} (1- \cos(2\beta t)).
\end{align*}
Setting $t= t_*/\gamma$ where $t_* \geqslant 1$ we find that there exists a constant $C>0$ so that
\begin{align}
\label{eqn:b1gamsmall}
D \geqslant \frac{4}{\gamma^2}(1-e^{t_*})^2 \{ 1 - C\gamma \}
\end{align}
for all $t_* \geqslant 1$, $0<\gamma \leqslant \gamma_s$, where $\gamma_s>0$ is sufficiently small and $C>0$ is independent of $\gamma$.  In a similar fashion, by decreasing $\gamma_s>0$ and $C>0$ if needed, we obtain the following bound
\begin{align}
\label{eqn:b2gamsmall}
-(G_{11}+ G_{22}) + \sqrt{(G_{11}+G_{22})^2 -4D} \leqslant \frac{8}{\gamma}(e^{t_*}-1) \{ 1+ C \gamma \}
\end{align}
satisfied for all $t_* \geqslant 1$ and $0< \gamma \leqslant \gamma_s$.  Combining~\eqref{eqn:b1gamsmall} with~\eqref{eqn:b2gamsmall} and adjusting $\gamma_s$ smaller if needed we obtain
\begin{align*}
\underline{\lambda}(2, t_*/\gamma) \geqslant \frac{|c_\gamma|} {\gamma}  (e^{t_*}-1) \frac{1-C\gamma}{1+C\gamma} \geqslant \frac{e^{t_*}-1}{8} \geqslant \frac{e^{t_*}}{16}
\end{align*}
for all $t_* \geqslant 1$ and $\gamma \leqslant \gamma_s$.

\subsubsection{Coupled oscillators}

For $a,d,b \in \RR$, let $\text{Tri}_j(a, d, b)$ denote the $j\times j$ tridiagonal matrix satisfying
\begin{align*}
(\text{Tri}_j(a,d,c))_{\ell m}= \begin{cases}
a & \text{ if } \,m=\ell-1\\
d & \text{ if } \,m=\ell \\
c & \text{ if } \, m=\ell+1\\
0 & \text{ otherwise}
\end{cases}
\end{align*}
and let $E_{\ell m}$ be the $j\times j$ matrix with $(\ell,m)$th entry equal to $1$ and all other entries equal to $0$.  Let the underlying matrices $\underline{A}$ and $\underline{\sigma}$ be given by
\begin{align*}
\underline{A}= \text{Tri}_j(1,0,-1) \qquad \text{ and } \qquad \underline{\sigma}= E_{11}.
\end{align*}
A short calculation shows that for $m=0,1,\ldots, j-1$ and $\ell=1,2,\ldots, j$
\begin{align*}
(\underline{A}^m \underline{\sigma})_{m+1, \ell} =E_{m+1, \ell} \qquad \text{ and } \qquad  (\underline{A}^m \underline{\sigma})_{z, \ell} = 0
\end{align*}
for $m+2\leq z \leq j-1$.  
Consequently, the matrix $\underline{A}_{\underline{\sigma}}$ has full rank and so $\underline{\lambda}(j, t) >0$ for all $t>0$.
The goal of this calculation will be to estimate $\underline{\lambda}(j,t)$ for $t\gg 1$.  We find it again convenient to diagonalize $\underline{A}$.

One can show that $\underline{A}$ has distinct eigenvalues $\lambda_\ell$, $\ell=1,2, \ldots, j$, given by
\begin{align}
\label{eqn:evalsad}
\lambda_\ell = 2 i \cos \Big(\frac{\ell\pi}{j+1}\Big)
\end{align}
with corresponding (right) eigenvectors $v_\ell$, $\ell=1,2, \ldots, j$, defined by
\begin{align*}
(v_\ell)_m= i^m\sin\Big( \frac{\ell m \pi}{j+1}\Big), \qquad m=1,2, \ldots, j.
\end{align*}
See, for example,~\cite{NoschesePasquiniReichel2013, SmithGBook1985}.  Moreover, using basic properties of trigonometric functions, it can be checked that the set of eigenvectors $\{ v_1, \ldots, v_j \}$ forms an orthogonal family with identical lengths (see~\cite{TatariHamadi2020})
\begin{align*}
|v_\ell|^2= \frac{j+1}{2},  \qquad \ell =1,2, \ldots, j.
\end{align*}
Therefore, we define an orthonormal family of (column) eigenvectors $\{ w_1, \ldots, w_j \}$ by $w_\ell=v_\ell/|v_\ell|$, $\ell=1,2,\ldots, j$.  Setting $P=[w_1 \,\, w_2\,\,  \ldots\,\,  w_j]$ and letting $P^{\ast}$ denote its Hermitian transpose, for any $x\in \RR^j_{\neq 0}$ we have by symmetry of $\sigma$ and antisymmetry of $\underline{A}$
\begin{align*}
\langle \underline{G}_t x, x \rangle  =\int_0^t |\underline{\sigma}^{\ast} e^{-v \underline{A}^{\ast}} x|^2 \, dv = \int_0^t |\underline{\sigma} e^{v \underline{A}} x|^2 \, dv = \int_0^t | E_{11} P \text{diag}(e^{v \lambda_1}, e^{v\lambda_2}, \ldots, e^{v\lambda_j}) P^{\ast} x|^2.
\end{align*}
Let $\xi= P^{\ast} x\in \mathbf{C}^j$ and notice that
\begin{align*}
\langle \underline{G}_t x, x \rangle = \int_0^t \Big| \sum_{\ell=1}^j (w_\ell)_1 \xi_\ell e^{\lambda_\ell v} \Big|^2 \, dv& = \sum_{\ell, m=1}^j (w_\ell)_1 \overline{(w_m)_1}  \xi_\ell \overline{\xi_m}  \int_0^t e^{v\lambda_\ell+ v\overline{\lambda_m}} \, dv\\
&=t \sum_{\ell} |(w_\ell)_1|^2 |\xi_\ell|^2  + \sum_{\ell \neq m} (w_\ell)_1 \overline{(w_m)_1}  \xi_\ell \overline{\xi_m}  \int_0^t e^{v\lambda_\ell+ v\overline{\lambda_m}} \, dv.\end{align*}
By definition of the eigenvalues $\lambda_m$ in~\eqref{eqn:evalsad}, for $\ell \neq m$, $\ell, m \in \{ 1,2, \ldots, j\}$,  we obtain
 \begin{align*}
\bigg| \int_0^t e^{v\lambda_\ell+ v\overline{\lambda_m}} \, dv\bigg| =\bigg| \frac{e^{2ti \cos (\frac{\ell \pi}{j+1})}-e^{2i t \cos (\frac{m \pi}{j+1})}}{2i \cos \big(\frac{\ell \pi}{j+1} \big)- 2i \cos \big(\frac{m\pi}{j+1} \big)}\bigg| \leqslant  \frac{1}{\big|\cos \big(\frac{\ell \pi}{j+1} \big) -  \cos \big(\frac{m\pi}{j +1} \big)\big|}. \end{align*}
Using the mean value theorem, we find that for $\ell \neq m$, $\ell, m \in \{ 1,2, \ldots, j \}$, there exists $p=p_{\ell m} $ strictly between $\ell \pi/(j+1)$ and $m \pi/(j+1)$ such that
\begin{align*}
\frac{1}{\big|\cos \big(\frac{\ell \pi}{j+1} \big) -  \cos \big(\frac{m\pi}{j +1} \big)\big|} = \frac{1}{\big|\frac{(\ell-m)\pi}{j+1} \sin(p)\big|}\leqslant \frac{(j+1)}{\pi}\frac{1}{\sin(\tfrac{\pi}{j+1})} =:c_j  .
\end{align*}
Hence, applying the bound above and using Cauchy-Schwarz we obtain
\begin{align*}
\bigg|\sum_{\ell \neq m} (w_\ell)_1 \overline{(w_m)_1}  \xi_\ell \overline{\xi_m}  \int_0^t e^{v\lambda_\ell+ v\overline{\lambda_m}} \, dv\bigg| \leqslant c_j \bigg(\sum_{\ell} |(w_\ell)_1 \xi_\ell| \bigg)^2 &\leqslant c_j \sum_{\ell} |(w_\ell)_1|^2 \sum_{\ell} |\xi_\ell |^2 \\
&= c_j |x|^2.
\end{align*}
Thus,
\begin{align*}
\langle G_t x, x \rangle \geqslant  t \sum_{\ell} |(w_\ell)_1|^2 |\xi_\ell|^2 -c_j |x|^2\geqslant \big( t \sin^2\big(\tfrac{\pi}{j+1}\big) - c_j\big)|x|^2
\end{align*}
and so we obtain
\begin{align*}
\underline{\lambda}(j,t) \geqslant t \sin^2\big(\tfrac{\pi}{j+1}\big) - c_j \end{align*}
for all $t>0$.

\subsubsection{Oscillators with some damping}
\label{sec:oscdamp}
We revisit the previous example, but this time we place damping on the first coordinate.  That is, we set
\begin{align*}
\underline{A}= \text{Tri}_j(1,0,-1) - E_{11} \qquad \text{ and } \qquad \underline{\sigma}= E_{11}.
\end{align*}
In this case, diagonalizing $\underline{A}$ as before seems challenging because an explicit expression for eigenvalues and eigenvectors is not known.  In order to analyze $\underline{\lambda}(j, t)$, we appeal to~\cref{prop:G_2}.

Note that, in a similar fashion to the previous example, one can readily check that $\underline{A}_{\underline{\sigma}}$ has full rank, so that $\underline{\lambda}(j,t) >0$ for all $t>0$.  We next seek to estimate
\begin{align}
\label{eqn:Lyap1}
\int_0^s | e^{-v \underline{A}^{\ast}} x|^2 \, dv.
\end{align}
Observe that the integrand in~\eqref{eqn:Lyap1} is precisely $|y(v)|^2$ where $y(v)$ is the solution of the following ODE on $\RR^j$ at time $v$:
\begin{align*}
\begin{cases}
\dot{y} = (E_{11} + \text{Tri}_j(1,0,-1)) y\\
y_0 = x \in \RR^j.
\end{cases}
\end{align*}
We claim that
\begin{align*}
|y(v)|^2 \geqslant d_j |x|^2 e^{r_j v} \text{ for all  }v\geqslant 0
\end{align*}
where $d_j , r_j >0$ are constants.  We will prove the claim using a convenient Lyapunov function.

To define our Lyapunov function, let $a_\ell$, $\ell=0,1,2, \ldots, j-1$, be positive constants to be determined and set $a_j=0$.  Define and function $V : \RR^j \rightarrow \RR$ by
\begin{align*}
V(y) =\frac{a_0}{2} |y|^2 - \sum_{i=1}^{j-1} a_i y_i y_{i+1}.
\end{align*}
We first pick a convenient form for the constants $a_i$; that is, we define
\begin{align*}
a_{i-1}= a_i + b_i, \,\,\,\, i=2, \ldots, j,
\end{align*}
for some positive constants $b_i$, $i=2, \ldots, j$, to be determined.  Note that, in particular, this means that
\begin{align*}
a_\ell = \sum_{i=\ell+1}^j b_i, \qquad \ell =1,\ldots, j,
\end{align*}
with the convention that the empty sum is zero.

First observe that we have the following explicit bound:
\begin{align}
\nonumber \big(\tfrac{a_0}{2}+\textstyle{\sum_{i=2}^j} b_i\big) |y|^2 &\geqslant  (\tfrac{a_0}{2} + \tfrac{a_1}{2})y_1^2+\sum_{i=2}^j  \big( \tfrac{a_{0}}{2}+ \tfrac{a_{i-1}}{2}+ \tfrac{a_{i}}{2}\big) y_i^2 \\
\label{eqn:LYINE1}
&\geqslant V(y)\\
\nonumber & \geqslant \big(\tfrac{a_0}{2}-\tfrac{a_1}{2} \big)y_1^2 + \sum_{i=2}^{j} \big( \tfrac{a_{0}}{2}- \tfrac{a_{i-1}}{2}- \tfrac{a_{i}}{2}\big) y_i^2\\
\nonumber & \geqslant \big(\tfrac{a_0}{2} - \textstyle{\sum}_{i=2}^j b_i \big) |y|^2.
\end{align}
In particular, we need to choose
\begin{align*}
a_0 > 2\sum_{i=2}^j b_i
\end{align*}
so that $V$ is nonnegative.

Next, observe that
\begin{align*}
\frac{d}{dt} V(y(t)) &= -a_1 y_1(t)y_2(t)+ \sum_{i=1}^{j} (a_{i-1} - a_{i}) (y_{i}(t))^2 +\sum_{i=1}^{j-2} (a_i- a_{i+1}) y_i(t) y_{i+2}(t)\\
&= -a_1 y_1(t)y_2(t)+ (a_0-a_1) y_1^2+  \sum_{i=2}^j b_i (y_i(t))^2 + \sum_{i=1}^{j-2} b_{i+1} y_i(t) y_{i+2}(t)\\
& \geqslant -a_1 y_1(t)y_2(t)+ (a_0-a_1- \tfrac{b_2}{2}) y_1^2+ (b_2- \tfrac{b_3}{2}) y_2^2 + (b_j - \tfrac{b_{j-1}}{2})y_j^2\\
&\qquad +\sum_{i=3}^{j-1} (b_i - \tfrac{b_{i+1}}{2}- \tfrac{b_{i-1}}{2}) (y_i(t))^2 .
\end{align*}
Pick $b_i=\log(i+1)$, $i=2,\ldots, j$, and notice by concavity we have
\begin{align*}
b_i > \frac{b_{i+1}}{2}+ \frac{b_{i-1}}{2}, \,\, i=2, \ldots, j.
\end{align*}
On the other hand, for any $\beta >0$, we have 
\begin{align*}
\frac{d}{dt} V(y(t)) &\geqslant  (a_0 - a_1- \tfrac{a_1^2}{2\beta}- b_2) y_1^2  + (b_2- \tfrac{b_3}{2} - \tfrac{\beta}{2}) y_2^2 + (b_j - \tfrac{b_{j-1}}{2})y_j^2\\
&\qquad +\sum_{i=3}^{j-1} (b_i - \tfrac{b_{i+1}}{2}- \tfrac{b_{i-1}}{2}) (y_i(t))^2
\end{align*}
Pick $\beta = \log(2)$, define $b_{j+1}=\log(j+2)$ and let
\begin{align*}
a_0= a_1 + \tfrac{a_1^2}{2\beta} +  b_2 + b_j - \tfrac{b_{j+1}}{2}- \tfrac{b_{j-1}}{2}. \end{align*}
Using the fact that the function $x\mapsto \log(x) -\log(x+1)/2 - \log(x-1)/2$ defined on $[2, \infty)$ is strictly decreasing and using~\eqref{eqn:LYINE1}, we obtain
\begin{align*}
\frac{d}{dt}V(y(t))\geqslant (b_j - \tfrac{b_{j+1}}{2}- \tfrac{b_{j-1}}{2}) |y(t)|^2\geqslant  \frac{(b_j - \tfrac{b_{j+1}}{2}- \tfrac{b_{j-1}}{2})}{a_0 + 2\sum_{i=2}^j b_i}V(y(t)).
\end{align*}
Hence, letting 
\begin{align}
\label{eqn:rcdef}
r_j = \frac{ (b_j - \tfrac{b_{j+1}}{2}- \tfrac{b_{j-1}}{2}) }{a_0+ 2\sum_{i=2}^j b_i} \qquad \text{ and } \qquad c_j = \frac{\frac{a_0}{2}- \sum_{i=2}^j b_i}{\frac{a_0}{2}+ \sum_{i=2}^j b_i},
\end{align}
we have that
\begin{align*}
|y(t)|^2 \geqslant |x|^2 c_j e^{r_j t} \,\,\, \text{ for all } t\geqslant 0.
\end{align*}
Thus,
\begin{align*}
\int_0^t |e^{-v\underline{A}^{\ast}} x|^2 \, ds \geqslant |x|^2 \frac{c_j}{r_j} (e^{r_j t}-1).
\end{align*}
Applying \cref{prop:G_2}, we obtain the bound for $t\geq 2$
\begin{align}
\label{eqn:oscdamp}
\underline{\lambda}(j,t) \geqslant \underline{\lambda}(j, t/2)\frac{c_j}{t r_j} (e^{\frac{r_j t}{2}}-1)\geq \frac{\underline{\lambda}(j, 1)}{t r_j} c_j  (e^{\frac{r_j t}{2}}-1)
\end{align}
\section{From finite to infinite dimensions}
\label{sec:FDTID}

The goal of this section is to move from the finite-dimensional setting in relation~\eqref{eqn:SDE} to an infinite-dimensional version of the system where the noise becomes infinite-dimensional.  In the context of the examples of the previous section, this means that if $n=jk$, where $j$ is the underlying dimension $j$ and $k$ is the dimension of the noise, then we take $k\rightarrow \infty$ leaving $j$ fixed.  Building off of the general analysis done with the modified gradients in \cref{sec:modgrad}, we will be able to extend the the Wang-Harnack inequality in \cref{thm:wang} to an infinite-dimensional version.  We will then use this infinite-dimensional version to conclude a quasi-invariance result as well as a ``time-infinity" Wang-Harnack inequality.  We conclude the section by revisiting some of the examples considered in \cref{sec:class} as they relate to their associated infinite-dimensional versions.

\subsection{The infinite-dimensional setting}
Fix a separable, infinite-dimensional Hilbert space $H$ with inner product and norm respectively denoted by 
\begin{align*}
\langle \cdot, \cdot \rangle_H \qquad \text{ and }\qquad \| \cdot \|_H=\sqrt{\langle \cdot,  \cdot  \rangle_H}.
\end{align*} Below, we make several slight abuses of notation to help connect with the finite-dimensional setting~\eqref{eqn:SDE} previously considered.  In particular, we will intentionally reuse the notations $Q, A, \sigma$, which were previously used to denote operators on $\RR^k, \RR^n, \RR^n$, respectively.  Here, in this section, they will be linear operators on the respective spaces $H$, $H^j$, $H^j$.  We let $W=H^j$ denote the separable Banach space with norm given by
\begin{align*}
\| X\|_W := \sqrt{\| X^1 \|_H^2+ \cdots + \| X^j \|_H^2  }
\end{align*}
where $X=(X^1, X^2, \ldots, X^j) \in H^j$.

Throughout this section we will employ the following:
\begin{Assumption}
\label{assump:2}
The mapping $Q:H\rightarrow H$ is a strictly positive symmetric bounded linear operator of trace-class.
\end{Assumption}
Note that under \cref{assump:2}, we can diagonalize $Q$ in $H$; that is, there is an orthonormal basis $\{ e_\ell \}_{\ell\in \N}$ of $H$ and real numbers $\alpha_\ell>0$ for which
\begin{align}
\label{eqn:Qdef}
Q e_\ell = \alpha_\ell e_\ell \,\, \text{  for all } \,\,  \ell\in \N.
\end{align}
Note that $Q$ being trace class translates to the summability condition
\begin{align}
\label{eqn:summability}
\sum_{\ell=1}^\infty \alpha_\ell  < \infty.
\end{align}
Again, offering slight abuses of notation, we assume that
\begin{align*}
B(t):=(B^1(t), B^2(t), \ldots, B^j(t))
\end{align*}
 where $B^i(t)$, $i=1,2,\ldots, j$, are independent, $Q$-Brownian motions on $H$.  Note that this is the same as supposing that each $B^i$ can be written as
\begin{align*}
B^i(t)= \sum_{\ell=1}^\infty \sqrt{\alpha_\ell} \beta_\ell^i(t) e_\ell
\end{align*}
where $\{ \beta_k^i \}_{k=1, i=1}^{\infty, j}$ is a collection of mutually independent standard, real-valued Brownian motions on $(\Omega, \mathcal{F}, \PP)$~\cite{DaPratoZabczykBook2014}.  The fact that each $B^i$ has continuous paths in $\| \cdot \|_H$ with probability one follows by path continuity of each $\beta_k^i$ and the summability condition~\eqref{eqn:summability}.

Recall that associated to $Q$ is the Hilbert space $H_Q$ consisting of elements $h$ such that 
\[ \sum_{\ell=1}^\infty \frac{\langle h, e_\ell\rangle_H^2}{\alpha_\ell}<\infty\]
equipped with the inner product $\langle h,k\rangle_Q:=\langle Q^{-1/2}h,Q^{-1/2}k\rangle_H$, where $Q^{-1/2}$ denotes the pseudo-inverse of $Q^{1/2}$. The space $H_Q$ is called the \emph{Cameron-Martin space} associated to the Gaussian measure $\mu=\operatorname{Law}(B^i(1))$, and enjoys various important analytic properties with respect to $\mu$. For example, we recall the Cameron-Martin-Maruyama ``quasi-invariance'' theorem: The measure $\mu^h$ defined by $\mu^h(\Gamma) := \mu(\Gamma+h)$ is mutually absolutely continuous with respect to $\mu$ if and only if $h\in H_Q$. If $h\notin H_Q$ then $\mu^h$ is singular with respect to $\mu$.
Let $\mathcal{H}_Q$ denote the Hilbert space $(H_Q)^j$ equipped with inner product coming from the product structure
\[ \langle X,Y\rangle_{\mathcal{H}_Q} := \langle X^1,Y^1\rangle_Q + \cdots + \langle X^j,Y^j\rangle_Q\]
and induced norm 
\[ \|X\|_{\mathcal{H}_Q} := \sqrt{\|X^1\|_Q^2 + \cdots + \|X^j\|_Q^2}
	= \|(I_{j\times j}\otimes Q^{-1/2}) X\|_W. \]

Letting $I_H$ denote the identity operator on $H$, we set
\begin{align}
\label{eqn:Asigdef}
A=(\underline{A} \otimes I_H) \qquad \text{ and } \qquad \sigma= (\underline{\sigma} \otimes I_H)
\end{align}
where $\underline{A}$ and $\underline{\sigma}$ are $j\times j$ real matrices.  In relation to~\eqref{eqn:SDE}, in this context we consider the integral equation on $W:=H^j$
\begin{align}
\label{eqn:SDEQ}
	X(t;X_0) &= X_0 + \int_0^t A X(s) \, ds + \sigma B(t), \,\,\, \, X_0 \in W.
\end{align}
Using a standard iteration procedure, it is not hard to show that for every $X_0 \in W$, relation~\eqref{eqn:SDEQ} has a unique solution which is a stochastic process $X(t)$ on $W$ with continuous paths in the norm $\| \cdot \|_W$.  Furthermore, the process $X(t)$ is Markov and we let $\mathcal{P}_t$ denote the corresponding Markov semigroup.

\subsection{Infinite-dimensional Wang-Harnack inequality}
Next, consider the projection operator $\pi_k : H\rightarrow H$ associated to the orthonormal basis $\{ e_k \}_{k\in \N}$ given by
\begin{align*}
\pi_k (h) := \pi_k\bigg( \sum_{\ell=1}^\infty \langle h, e_\ell\rangle e_\ell\bigg) = \sum_{\ell=1}^k \langle h, e_\ell\rangle e_\ell,
\end{align*}
and then define $\Pi_k : W\rightarrow W$ as
\begin{align*}
\Pi_k(h_1 , h_2, \ldots, h_j) := (\pi_k(h_1), \ldots, \pi_k(h_j)).
\end{align*}
We observe that, by~\eqref{eqn:Asigdef} and~\eqref{eqn:SDEQ}, the process $t\mapsto \Pi_k X(t)$ satisfies the finite-dimensional integral equation
\begin{align*}
	\Pi_k X(t;X_0)= \Pi_k X_0 + \int_0^t A \Pi_k X(s) \, ds + \sigma \Pi_k B_t.
\end{align*}
Let $\mathcal{P}_t^k$ denote the Markov semigroup corresponding to $\Pi_k  X(t)$. 

\begin{Proposition}
\label{prop:project}
	Suppose that $\underline{A}_{\underline{\sigma}}$ as in~\eqref{eqn:kalsub} is of full rank.  Let $\alpha >1$.  Then for all bounded, measurable $\varphi: W\rightarrow \RR$, $t>0$, $X_0, Y_0 \in W$, and  $k\in\mathbf{N}$ we have
\begin{align}
\label{eqn:Wscaled}
	(\mathcal{P}_t^k \varphi (\Pi_k X_0))^\alpha
	\leqslant \mathcal{P}_t^k \varphi^\alpha(\Pi_kY_0)\exp\bigg(\frac{\alpha}{\alpha-1} \frac{ \|  (I_{j\times j} \otimes Q^{-1/2})\Pi_k(X_0-Y_0) \|_W^2}{2\underline{\lambda}(j,t)} \bigg).
\end{align}
\end{Proposition}
\begin{proof} This follows after combining \cref{thm:wang} with \cref{prop:tensor}. 
\end{proof}
The following infinite-dimensional Wang-Harnack inequality follows immediately from the previous result by allowing $k\to\infty$.
\begin{Theorem}
	\label{thm:infwang}
Suppose that $\underline{A}_{\underline{\sigma}}$ as in~\eqref{eqn:kalsub} is of full rank and that $X_0, Y_0$ are such that $X_0-Y_0\in\mathcal{H}_Q$. Let $\alpha>1$. Then for all bounded, measurable $\varphi: W\rightarrow \RR$ and $t>0$ we have
\begin{align*}
	(\mathcal{P}_t \varphi ( X_0))^\alpha \leqslant \mathcal{P}_t \varphi^\alpha(Y_0)\exp\bigg(\frac{\alpha}{\alpha-1} \frac{\|X_0-Y_0\|_{\mathcal{H}_Q}^2}{{2}\underline{\lambda}(j,t)} \bigg).
\end{align*}
\end{Theorem}
\begin{proof}
One need only show that for any $X_0 \in W$,
\[
\Pi_k X(t) \xrightarrow{a.s.} X(t)
\]
as $k\rightarrow \infty$, and this follows immediately by uniqueness of solutions and the definition of $\Pi_k$.
\end{proof}

The equivalence of Wang-Harnack inequalities with integrated Harnack inequalities has been established in \cite{DriverGordina2009,BaudoinGordinaMelcher2013,Gordina2017}, along with how these estimates imply quasi-invariance results of the following kind.

\begin{Theorem}
	\label{thm:quasi}
	Fix $t>0$. The measure $\mu_t:=\operatorname{Law}(X_t(0))$ is quasi-invariant under translations by elements of $\mathcal{H}_Q$, that is, for $X_0\in \mathcal{H}_Q$ the measure $\mu_t(X_0):=\operatorname{Law}(X_t(X_0))$ is mutually absolutely continuous with respect to $\mu_t$.
	Moreover, for all $p>1$
	\[ \left\|\frac{d\mu_t(X_0)}{d\mu_t}\right\|_{L^p(W,\mu_t)} \leqslant \exp\left(\frac{1+p}{{{2}\underline{\lambda}(j,t)}}  \|  X_0\|_{\mathcal{H}_Q}^2 \right).\]
\end{Theorem}

\begin{Remark}
In each of the examples considered in \cref{sec:class} we saw that $\underline{A}_{\underline{\sigma}}$ has full rank.  Thus \cref{thm:quasi} applies to the infinite-dimensional versions of these equations when the noise becomes infinite dimensional ($k\rightarrow \infty$) as described above.
\end{Remark}


%

\subsection{Mutual absolute continuity at time infinity}
Under further conditions, we will use the results of~\cref{sec:modgrad} to deduce that the laws of the infinite-dimensional Markov process $X(t)$ solving~\eqref{eqn:SDEQ} started from two initial conditions $X_0$ and $Y_0$ in $W$ become mutually absolutely continuous at ``time infinity". This follows from a strengthening of \cref{thm:infwang}, noting that that result only applies to sufficiently smooth initial conditions in $W$ at a finite time $t>0$.  Formally taking $t,k\rightarrow \infty$ in the bound~\eqref{eqn:Wscaled} suggests this should hold, but more care needs to be taken when we pass to the limit.

To this end, for $k\in \N$ define
\begin{align*}
\Delta_k:=(I_W-\Pi_k).
\end{align*}
\begin{Assumption}
\label{assump:3}
The matrix $\underline{A}_{\underline{\sigma}}$ in~\eqref{eqn:kalsub} is of full rank and the constant $\underline{\lambda}(j, t)>0$ as in~\eqref{def:lambda}  satisfies the following conditions:
\begin{itemize}
\item[(i)] $\underline{\lambda}(j,t)\rightarrow \infty$ as $t\rightarrow \infty$.
\item[(ii)]  Define by (i) a sequence $\{ t_k \}$ of positive times $t_k\rightarrow \infty$ such that $\underline{\lambda}(j, t_k)= \alpha_k^{-1}$ for all $k$ where we recall that the $\alpha_k$ are as in~\eqref{eqn:Qdef}.  Then there exists a $p>1$ such that for any $X_0 \in W$ we have that
	\begin{align*}
\E\| \Delta_k X(t_k; X_0)\|^p_{W} \rightarrow 0.
\end{align*}
\end{itemize}
\end{Assumption}

Under \cref{assump:3}, we can prove our main result at ``time infinity".  However, we first establish a natural condition which ensures \cref{assump:3}(ii) is satisfied.
\begin{Proposition}
\label{prop:conditiontimeinf}
Suppose that the underlying matrix $\underline{A}$ as in~\eqref{eqn:Asigdef} is such that $\langle \underline{A} x , x \rangle_{\RR^j} \leqslant 0$ for all $x\in \RR^j$.
Suppose, furthermore, that \cref{assump:3}(i) is satisfied and that the sequence of times $\{ t_k \}$ defined by $\lambda(j, t_k) = \alpha_k^{-1}$, $k\in \N$, satisfies the condition \begin{align}\label{eqn:growthlam}
t_k\sum_{\ell=k+1}^\infty \alpha_\ell \rightarrow 0 \text{ as } k\rightarrow \infty.
\end{align}
Then \cref{assump:3}(ii) is satisfied with $p=2$.
\end{Proposition}

\begin{Remark}
\label{rem:propmix}
Suppose that $\underline{\lambda}(j, t) \geqslant  c_j t^3 -d_j$ for all $t\geqslant 0$ for some constants $c_j,d_j >0$ and that $\alpha_k = 1/k^2$.  Then in this case,
\begin{align*}
t_k \lesssim k^{2/3}
\end{align*}
 so that
\begin{align}
\label{eqn:timecond}
t_k \sum_{\ell = k+1}^\infty \alpha_\ell  \lesssim \frac{1}{k^{1/3}}  \rightarrow 0 \text{ as } k\rightarrow \infty.
\end{align}
In short, the smallest positive eigenvalue $\underline{\lambda}(j, t)$ must grow sufficiently fast as $t\rightarrow \infty$ with respect to the sequence $\alpha_k$ so that~\eqref{eqn:growthlam} is satisfied.
\end{Remark}
%

\begin{proof}[Proof of~\cref{prop:conditiontimeinf}]
Let $Y_k(t):= \Delta_k X(t; X_0)$ and define
\begin{align*}
\tau_{k,m} = \inf \{ t \geqslant 0 \, : \, \| Y_k(t) \|_W \geqslant m \}.
\end{align*}
Set $\tau_{k,m}(t) = t \wedge \tau_{k,m}$ and observe that, by It\^{o}'s formula and the condition $\langle \underline{A} x, x \rangle_{\RR^j} \leqslant 0$ we have
\begin{align*}
\E \| Y_k(\tau_{k,m}(t)) \|_W^2 &= \| Y_k(0)\|_W^2 + 2 \E \int_0^{\tau_{k,m}(t)} \sum_{\ell=1}^j \langle (A Y_k(s))_\ell, (Y_k(s))_\ell\rangle_H  \, ds \\
& \qquad + \sum_{\ell_1, \ell_2=1}^j  \underline{\sigma}_{\ell_1 \ell_2} ^2 \sum_{\ell=k+1}^\infty \alpha_\ell \E \tau_{k,m}(t)\\
& \leqslant  \| Y_k(0)\|_W^2+ \sum_{\ell_1, \ell_2=1}^j  \underline{\sigma}_{\ell_1 \ell_2} ^2 \sum_{\ell=k+1}^\infty \alpha_\ell  t.
\end{align*}
Taking $m\rightarrow \infty$ and then plugging in $t=t_k$, we arrive at the following estimate
\begin{align*}
\E \| Y_k(t_k) \|_W^2 &\leqslant \| Y_k(0)\|_W^2 +\sum_{\ell_1, \ell_2=1}^j  \underline{\sigma}_{\ell_1 \ell_2} ^2 t_k \sum_{\ell=k+1}^\infty \alpha_\ell   .
\end{align*}
Applying the hypothesis~\eqref{eqn:timecond}, we conclude the result.
\end{proof}

We are now prepared to state and prove our uniqueness result.

\begin{Corollary}[Uniqueness of stationary distributions via Wang-Harnack]
\label{thm:asf}
Suppose that~\cref{assump:3} is satisfied and $\mu$ and $\nu$ are stationary distributions for the Markov process $X(t)$.  Then for any $\alpha >1$ there exists a constant $C=C(\mu, \nu, \alpha)>0$ such that
\begin{align}
\label{eqn:mainbound}
	(\mu(\varphi))^\alpha \leqslant C\nu(\varphi^\alpha) 
\end{align}
for all $\varphi: W\rightarrow [0, \infty)$ bounded, measurable.  Consequently, there is at most one stationary distribution corresponding to $X(t)$.
\end{Corollary}

\begin{proof}
Note that the final conclusion in the result follows immediately by symmetry of the bound~\eqref{eqn:mainbound} and ergodic decomposition, for any two such stationary distributions must be mutually absolutely continuous.  Let $\alpha >1$.
In order to prove~\eqref{eqn:mainbound}, it suffices to prove the bound for all $\varphi:W\rightarrow [0, \infty)$ which are bounded, Lipschitz and such that $\varphi^\alpha$ is also Lipschitz; that is, $\| \varphi \|_{L^\infty} < \infty$,
\begin{align*}
\| \varphi\|_{\text{Lip}}:= \sup_{\substack{X,Y \in W\\ X\neq Y}} \frac{| \varphi(X)-\varphi(Y)|}{\|X-Y\|_W} < \infty \qquad \text{ and } \qquad \| \varphi^\alpha \|_{\text{Lip}} < \infty.
\end{align*}
Let $\epsilon \in (0,1)$ and pick $R=R(\mu, \nu, \alpha)>0$ large enough so that $B_R:= \{ \|X\|_W < R \}$ satisfies
\begin{align*}
 \mu(B_R^c) < \tfrac{\epsilon}{2} \qquad \text{ and } \qquad    \nu(B_R^c) < \tfrac{\epsilon}{2}.\end{align*}
Then by invariance and Jensen's inequality we have that
\begin{align*}
\mu(\varphi)^\alpha = \bigg[\int \mu(dX) \mathcal{P}_{t_k} \varphi(X)\bigg]^\alpha &\leqslant \int  \mu(dX) (\mathcal{P}_{t_k} \varphi(X))^\alpha \leqslant  \int_{B_R} \mu(dX) (\mathcal{P}_{t_k} \varphi(X))^\alpha + \tfrac{\epsilon}{2} \| \varphi \|_{L^\infty}^\alpha. \end{align*}
Multiplying the previous inequality by $\nu(W)=1$ and writing this as $\nu(W)= \nu(B_R)+ \nu(B_R^c)$ gives
\begin{align*}
\mu(\varphi)^\alpha &\leqslant  \nu(B_R) \int_{B_R} \mu(dX) (\mathcal{P}_{t_k}\varphi(X))^\alpha + \epsilon \| \varphi\|_{L^\infty}^\alpha\\
  &=  \int_{B_R} \int_{B_R}(\mathcal{P}_{t_k}\varphi(X))^\alpha \mu(dX) \nu(dY)+ \epsilon   \| \varphi\|_{L^\infty}^\alpha.
\end{align*}
In order to control the double integral above, observe that
\begin{align*}
\mathcal{P}_{t_k}\varphi(X)= \mathcal{P}_{t_k}^k \varphi( \Pi_k X) + \E (\varphi(X(t_k; X)) - \varphi(\Pi_k X(t_k; X)))
\end{align*}
and
\begin{align*}
|\E (\varphi(X(t_k;X)) - \varphi(\Pi_k X(t_k;X))) | \leqslant 2\| \varphi \|_{L^\infty} \wedge\big\{ \| \varphi \|_{\text{Lip}}  (\E\| \Delta_k X(t_k, X)\|^p)^{1/p}\big\}=:C_k(\varphi, X).
\end{align*}
By \cref{assump:3}, we have that $C_k(\varphi, X) \rightarrow 0$ pointwise in $X$ as $k\rightarrow \infty$.
Hence applying~\cref{prop:project} we have
\begin{align*}
 &\int_{B_R} \int_{B_R}(\mathcal{P}_{t_k}\varphi(X))^\alpha \mu(dX) \nu(dY)\\
 &\leqslant 2^\alpha \int_{B_R} \int_{B_R}(\mathcal{P}_{t_k}^k\varphi(\Pi_k X))^\alpha \mu(dX) \nu(dY)+ 2^\alpha \int_{B_R} \mu(dX) C_k(\varphi, X) \\
 & \leqslant 2^\alpha C(\alpha, R) \int_{B_R} \int_{B_R}(\mathcal{P}_{t_k}^k\varphi^\alpha(\Pi_k Y)) \mu(dX) \nu(dY)+ 2^\alpha \int_{B_R} \mu(dX) C_k(\varphi, X) \\
 & \leqslant  2^\alpha C(\alpha, R) \int_{B_R} \int_{B_R}(\mathcal{P}_{t_k}\varphi^\alpha(Y)) \mu(dX) \nu(dY)+ 2^\alpha \int_{B_R} \mu(dX) C_k(\varphi, X)\\
  & \qquad + 2^{2\alpha} C(\alpha, R) \int_{B_R} \nu(dY) C_k(\varphi^\alpha, Y).  \\
  & \leqslant 2^\alpha C(\alpha, R) \nu(\varphi^\alpha)+ 2^\alpha \int_{B_R} \mu(dX) C_k(\varphi, X) + 2^{2\alpha} C(\alpha, R) \int_{B_R} \nu(dY) C_k(\varphi^\alpha, Y).
  \end{align*}
 Taking $k\rightarrow \infty$ using \cref{assump:3} and the bounded convergence theorem and then letting $\epsilon \rightarrow 0$ we obtain the desired bound~\eqref{eqn:mainbound} for all $\varphi: W\rightarrow [0, \infty)$ bounded, Lipschitz with $\varphi^\alpha$ Lipschitz. This finishes the proof.
\end{proof}

We now revisit the linear kinetic Fokker Planck equation and the degenerately damped oscillators considered in~\cref{sec:class}.

\subsubsection{Linear kinetic Fokker-Planck}
\label{sec:KFP2}

In the setting of~\cref{sec:KFP} in Case 1 ($\gamma \gg 1$), we recall $j=2$ and that we obtained the following bound
\begin{align}
\label{eqn:KFPlb}
\underline{\lambda}(2, t) \geqslant  \frac{e^{t/\gamma}}{16}
\end{align}
for all $t\geqslant \gamma \geqslant \gamma_L$.  Fixing $\gamma \geqslant \gamma_L$ and a summable sequence $\{\alpha_k\}_{k\in \N}$, we observe by~\eqref{eqn:KFPlb} that the sequence $\{ t_k\}$ in \cref{assump:3} satisfies, for all $k$ large enough,
\begin{align*}
t_k \leqslant \gamma \log (\alpha_k^{-1}) + \gamma \log(16).
\end{align*}
Furthermore, in this context it is clear that  $\langle \underline{A} x, x\rangle_{\RR^2}\leqslant 0$.  Thus by \cref{prop:conditiontimeinf}, for \cref{assump:3} to be satisfied, it suffices that $\{ \alpha_k \}$ satisfy
\begin{align}
\label{eqn:FPnoise}
\log (\alpha_k^{-1}) \sum_{\ell=k+1}^\infty \alpha_\ell \rightarrow 0 \text{ as } k \rightarrow \infty.
\end{align}
Clearly, condition~\eqref{eqn:FPnoise} is satisfied if the noise decays fast enough at large scales, e.g. $\alpha_k \sim 1/k^p$ or even $\alpha_k \sim 1/(k\log (k)^p)$ for some $p>1$.  In particular, under~\eqref{eqn:FPnoise}, uniqueness of stationary distributions follows from~\cref{thm:asf}.  The same conclusion also holds in the case when $\gamma \approx 0$.

\subsubsection{Oscillators with some damping}
In the setting of~\cref{sec:oscdamp}, we obtained the bound
\begin{align*}
\underline{\lambda}(j, t) \geqslant \underline{\lambda}(j, 1) \frac{c_j}{r_j} (e^{r_j (t-1)} -1)
\end{align*}
for all $t\geqslant 1$~\eqref{eqn:oscdamp} where $r_j$ and $c_j$ are positive constants defined in~\eqref{eqn:rcdef}.  Because $\underline{\lambda}(j,1) >0$ and $\langle \underline{A} x,x\rangle_{\RR^j}\leqslant 0$, a similar argument to the one used in~\cref{sec:KFP2} implies that, provided (\ref{eqn:FPnoise}) is satisfied,
there is at most one stationary distribution for the corresponding infinite-dimensional process~\eqref{eqn:SDEQ} by~\cref{prop:conditiontimeinf} and~\cref{thm:asf}.

\bibliographystyle{plain}
\bibliography{linear}

\end{document}